\documentclass[12pt]{amsart}
\usepackage{amsmath,amssymb,enumerate}  
\usepackage[margin=1in]{geometry}

\usepackage{amsmath}

\usepackage{tikz}
\usetikzlibrary{arrows}
\usetikzlibrary{positioning}

\usepackage{empheq}   

\newtheorem{theorem}{Theorem}
\newtheorem{lemma}{Lemma}

\newtheorem{corollary}{Corollary}

\DeclareMathOperator{\mex}{mex}

\newcommand{\raf}[1]{(\ref{#1})}

\def\U{\mathcal{U}}
\def\cK{\mathcal{K}}
\def\cH{\mathcal{H}}
\def\cA{\mathcal{A}}
\def\cB{\mathcal{B}}
\def\cC{\mathcal{C}}
\def\cF{\mathcal{F}}
\def\cH{\mathcal{H}}
\def\G{\mathcal{G}}

\def\P{\mathcal{P}}
\def\N{\mathcal{N}}

\def\ZZP{\mathbb{Z}_{+}}

\def\ga{\alpha}

\def\go{\omega}

\def\hH{{\widetilde{\mathcal{H}}}}
\def\hx{{\widetilde{x}}}

\begin{document}

\title{Sprague-Grundy Function of Matroids and Related Hypergraphs}

\author{Endre Boros}
\address{MSIS and RUTCOR, RBS, Rutgers University,
100 Rockafeller Road, Piscataway, NJ 08854}
\email{endre.boros@rutgers.edu}

\author{Vladimir Gurvich}
\address{National Research University Higher School of Economics, Russian Federation}
\email{vgurvich@hse.ru}

\author{Nhan Bao Ho}
\address{Department of Mathematics and Statistics, La Trobe University, Melbourne, Australia 3086}
\email{nhan.ho@latrobe.edu.au, nhanbaoho@gmail.com}

\author{Kazuhisa Makino}
\address{Research Institute for Mathematical Sciences (RIMS)
Kyoto University, Kyoto 606-8502, Japan}
\email{e-mail:makino@kurims.kyoto-u.ac.jp}

\author{Peter Mursic}
\address{MSIS and RUTCOR, RBS, Rutgers University,
100 Rockafeller Road, Piscataway, NJ 08854}
\email{peter.mursic@rutgers.edu}





\thanks{The authors thank Ilya Bogdanov and D\"om\"ot\"or P\'alv\"olgyi for their helpful suggestions cited in Section
\ref{s-size}, and the referees for careful reading and helpful remarks. The authors also thank Rutgers University and RUTCOR for the support
to meet and collaborate in September-October 2016, March 2017 and 2018.
The authors thank RIMS for the support to meet and collaborate in January 2017.
The second author was partially funded by
the Russian Academic Excellence Project '5-100'.
The fourth author was partially supported by JSPS KAKENHI Grant
Number JP24106002, JP25280004, JP26280001, and JST CREST Grant Number JPMJCR1402, Japan.}

\begin{abstract}
We consider a generalization of the classical
game of $NIM$ called hypergraph $NIM$.
Given a hypergraph  $\cH$
on the ground set  $V = \{1, \ldots, n\}$ of $n$ piles of stones,
two players alternate in choosing a hyperedge
$H \in \cH$ and strictly decreasing all piles $i\in H$.
The player who makes the last move is the winner.
In this paper we give an explicit formula that describes the Sprague-Grundy function of hypergraph $NIM$
for several classes of hypergraphs.
In particular we characterize all $2$-uniform hypergraphs
(that is  graphs) and all matroids for which the formula works. We show that all self-dual matroids are included in this class.
\end{abstract}

\maketitle

\section{Introduction}
\label{s1}


In the classical game of $NIM$ there are
$n$  piles of stones and two players move alternating.
A move consists of choosing a nonempty pile and taking some positive number of stones from it.
The player who cannot move is the loser.
Bouton \cite{Bou901} analyzed this game and described the winning strategy for it.

In this paper we consider the following generalization of $NIM$.
Given a hypergraph $\cH \subseteq 2^V$, where $V = \{1, \dots , n\}$, two players alternate in choosing
a hyperedge $H \in \cH$ and strictly decreasing all piles $i\in H$. Different piles can be decreased by different (positive) amounts.
We assume in this paper that $\cH\not=\emptyset$ and $\emptyset \not\in \cH$ for all considered hypergraphs $\cH$.
In other words, every move strictly decreases some of the piles.
Similarly to $NIM$, the player who cannot move is losing.
This game is called $NIM_\cH$ and some special cases of it were considered in \cite{BGHMM15, BGHMM16}.

$NIM_\cH$ is an impartial game.
In this paper we do not need to immerse in the theory of impartial games.
We will need to recall only a few basic facts to explain and motivate our research.
We refer the reader to \cite{Sie13} for more details; see also \cite{Alb07,BCG01-04}.

It is known that the set of positions of an impartial game can
uniquely be partitioned into sets of $\P$ and $\N$ positions
(in which, respectively, the previous and the next player can win).
Every move from a $\P$ position goes to an $\N$ position, while
from an $\N$ position always there exists a move to a $\P$ position.
This partition shows how to win the game, whenever possible.

The so-called Sprague-Grundy (SG) function $\G_\Gamma$ of an impartial game $\Gamma$ is
a refinement of the above partition.
Namely, $\G_\Gamma(x)=0$ if and only if $x$ is a $\P$ position.
The notion of the SG function for impartial games was introduced by Sprague and Grundy
\cite{Spr35, Spr37, Gru39} and it plays a fundamental role
in determining the $\N-\P$  partition of \emph{disjunctive sums} of impartial games.

Finding a formula for the SG function of an impartial game remains a challenge.
Closed form descriptions are known only for some special classes of impartial games.
We recall below some known results.
The purpose of our research is to extend these results and
to describe classes of hypergraphs for which
we can provide a closed formula for the SG function of $NIM_\cH$.
To follow our proofs, we need to recall the precise definition of the SG function,
which we will do in Section \ref{s2}.

The game $NIM_\cH$ is a common generalization of several families of impartial games considered in the literature.
For instance, if $\cH=\{\{1\},\dots ,\{n\}\}$ then $NIM_\cH$ is the classical $NIM$,
which was analyzed and solved by Bouton \cite{Bou901}.
The case of $\cH=\{S\subseteq V\mid 1\leq |S|\leq k\}$, where $k<n$, was considered by Moore \cite{Moo910}.
He characterized the $\P$ positions of these games, that is those with SG value $0$.
Jenkyns and Mayberry \cite{JM80} described also the
set of positions in which the SG value is $1$ and
provided an explicit formula for the SG function in the subcase of $k = n-1$.
This result was extended in \cite{BGHM15}.
In \cite{BGHMM15} the game $NIM_\cH$ was considered
in the case of $\cH=\{S\subseteq V \mid |S|=k\}$
and the corresponding SG function was determined when  $2k\geq n$.
Let us also mention NIM played on a simplicial complex studied by \cite{ES96}.
These games are hypergraph NIM games with hypergraphs $\cH$ that are independence systems
(without the empty set; that is if $\emptyset\neq X\subset Y\in \cH$ then we have $X\in\cH$).
For such games $\P$ and $\N$ positons were characterized under some conditions in \cite{ES96}.

To state our main result we need to introduce some additional notation.
We denote by $\ZZP$ the set of nonnegative integers and
use $x\in\ZZP^V$ to describe a position, where coordinate $x_i$ denotes
the number of stones in pile $i \in V$.
Given a hypergraph $\cH$ and position $x \in \ZZP^V$, we denote by $\G_\cH(x)$ the SG value of $x$ in $NIM_\cH$.
The height $h_{\cH}(x)$ was defined in \cite{BGHMM15} as the maximum number of consecutive moves
that the players can make in $NIM_\cH$ starting from position $x$.

To a position $x \in \ZZP^V$ of $NIM_\cH$ let us associate the following quantities:
\begin{subequations}\label{e-myv}
\begin{align}
m(x) &=\min_{i\in V} x_i \label{e-m}\\
y^{}_\cH(x) &=h_\cH(x-m(x)e)+1 \label{e-y}\\
v^{}_\cH(x) &=\binom{y^{}_\cH(x)}{2} +
\left( \left(m(x)-\binom{y^{}_\cH(x)}{2}-1\right)\mod y^{}_\cH(x)\right),\label{e-v}
\end{align}
\end{subequations}
where $e$ is the $n$-vector of full ones. Finally, we define
\begin{empheq}[left={f(x)=\empheqlbrace}]{align}
    h_\cH(x)    &\quad \text{ if } m(x) \leq \binom{y^{}_\cH(x)}{2}    \label{e-JM-I}\\
    v^{}_\cH(x) &\quad \text{ otherwise.}                              \label{e-JM-II}
\end{empheq}

With this notation the results of \cite{BGHM15,BGHMM15,JM80} can be stated as follows:
the SG function of the considered games is
defined by \eqref{e-JM-I}-\eqref{e-JM-II}, that is, $\G = \U$.
It was a surprise to see that the ``same'' formula works for seemingly very different games.
In view of this, we call the expression \eqref{e-JM-I}-\eqref{e-JM-II}
the \emph{JM formula}, in honor of the results of Jenkyns and Mayberry \cite{JM80}.
We call a hypergraph $\cH$ a \emph{JM hypergraph} if this formula describes the SG function of $NIM_\cH$.

It is difficult to give an intuitive explanation for the above 
formula.
It was observed numerically that for many positions of hypergraph NIM games the SG value is equal to the height.
Yet, for some positions it is much less and shows surprising periodicity \cite{BGHM15}.
For JM hypergraphs this periodicity is explained by the above formula.
In particular, 
it is important to note that  $\G_\cH(x) = h_\cH(x)$  for every position  $x$  with $m(x)=0$.

Although the formula looks the same for all hypergraphs,
but it contains height  $h_\cH$  and, hence, the actual values depend on 
$\cH$.
The $\P$ positions of {\sc Nim}$_\cH$ are exactly the ones for which $h_\cH(x - m(x)e) = 0$.
This condition is easy to check even though computing  $h_\cH(x)$ may be computationally hard  \cite{BGHMM16}.

Given a hypergraph $\cH\subseteq 2^V$ and a subset $S\subseteq V$,
we denote by $\cH_S$ the \emph{induced subhypergraph}, defined as
\[
\cH_S=\{H\in\cH\mid H\subseteq S\}.
\]
A set $T \subseteq V$ is called a {\em transversal} if  $T \cap H \neq \emptyset$ for all $H\in\cH$.
A hypergraph $\cH$ is called {\em transversal-free} if no hyperedge $H \in \cH$ is a transvesal of $\cH$.
Finally, we say that $\cH$ is \emph{minimal transversal-free} if it is transversal-free, while
every nonempty proper induced subhypergraph of it is not.
A hypergraph $\cH$  is $k$-\emph{uniform} if $|H|=k$ for all $H \in \cH$.

We assume that the readers are familiar with the notion of matroids, see, e.g., \cite{Wel76,Wel2010}.
A {\em matroid} hypergraph $\cH \subseteq 2^V$ is formed by the family of  bases of a matroid on the ground set $V$.
It is {\em self-dual} if  $V \setminus H \in \cH$  for all $H \in \cH$, that is, if the corresponding matroid is self-dual.
Let us remark that in some papers self-dual matroids are called \emph{identically self-dual}.

In this paper we provide some necessary and some sufficient conditions
for a hypergraph to be JM. We summarize our main results as follows.

\smallskip

\begin{itemize}
\item[(i)] A JM hypergraph is minimal transversal-free.
\item[(ii)] A graph (that is, a $2$-uniform hypergraph) is JM if and only if it is  connected and minimal transversal-free.
We provide a complete list of JM graphs.
\item[(iii)] A matroid hypergraph is JM if and only if it is transvesal-free.
This implies that all self-dual matroid hypergraphs are JM.
\item[(iv)] Hypergraphs defined by connected $k$-edge subgraphs of a given graph are JM under certain conditions.
\item[(v)] For every integer $k$, the number of vertices of a $k$-uniform JM hypergraph is bounded by $k\binom{2k}{k}$.
\end{itemize}

\smallskip

For instance, $\binom{V}{k}=\{H\subseteq V \mid |H|=k\}$
is a self-dual matroid hypergraph if $n=2k$.
This example shows that (iii) generalizes the main result of \cite{BGHMM15}.
Another example for a self-dual matroid with $n=2k$ is the hypergraph
$\cH_{2^k}=\{H\subseteq V\mid |H\cap\{i,i+k\}|=1,~ 1\leq i\leq k\}$,
that is, the family of $2^k$ minimal transversals of a family of $k$ pairs.
It was proved in \cite{BW71} that any self-dual matroid
on $n=2k$ elements must have at least $2^k$ bases.
Thus, the latter construction is extremal in this respect.

We remark that \cite{BW71} showed also the existence of self-dual matroids on $n=2k$ elements
whenever certain type of symmetric block designs exists on $k$ points.
Since many families of such block designs are known,
the above cited result shows that numerous other families of self-dual matroids (and JM hypergraphs) exist.

For (iv) we can mention the following circulant hypergraphs defined by consecutive $k$ edges of simple cycles on $n=2k$ or $n=2k+1$ vertices.
Another example is defined by connected $k$-edge subgraphs of a rooted tree, where the root has degree $k+1$ and
each of the subtrees connected to the root have exactly $k$ edges; see Section \ref{s4} for precise definitions and details.

\medskip

Let us include here a small example to show how these quantities are computed and used.
Consider $V=\{1,2,3,4\}$ and the $2$-uniform hypergraph $\cH=\{\{1,2\},\{1,4\},\{2,3\},\{3,4\}\}$.
It is easy to see that 
$h_\cH(x_1,x_2,x_3,x_4)=\min\{x_1+x_3,x_2+x_4\}$ for any position $x\in\ZZP^V$.
Furthermore, consider positions $a=(1,2,4,2)$ and $b=(4,5,7,5)$ 
of $NIM_\cH$.
For $a$ we have $m(a)=1$ and, hence, $a - m(a)e=(0,1,3,1)$;
for $b$ we have $m(b)=4$ and, hence, $b - m(b)e=(0,1,3,1)$.
For both positions we have $h_\cH(a-m(a)e) = h_\cH(b-m(b)e)=\min\{0+3,1+1\}=2$ and, thus by \eqref{e-y},
we obtain $y_\cH(a)=y_\cH(b)=3$.

Since for position $a$ we have $1 = m(a) \leq \binom{y_\cH(a)}{2}=\binom{3}{2}=3$,
we can conclude by \eqref{e-JM-I} that \[\U_\cH(a)=h_\cH(a)=\min\{1+4,2+2\}=4.\]
For position $b$ we have $4=m(b)>\binom{y_\cH(b)}{2}=\binom{3}{2}=3$, thus, \eqref{e-JM-II} implies that
\[\U_\cH(b)=v_\cH(b)= \binom{3}{2} + \left( \left( 4-\binom{3}{2}-1\right) \mod 3\right) = 3 < h_\cH(b)=\min\{4+7,5+5\}=10.\]

We will see in Section \ref{s4} that $\cH$ is a JM hypergraph, thus,
the above computed values for these positions are in fact their SG values.
Hence, $a$ and $b$ are both $\N$ positions.
Furthermore, $\cH$ is a minimal transversal free hypergraph, and thus,
deleting a pile with the minimum size creates a subhypergraph with a transversal edge.
For both $a$ and $b$ coordinate $1$ is the unique smallest pile, and therefore we consider $S=\{2,3,4\}$.
The induced subhypergraph is $\cH_S=\{\{2,3\},\{3,4\}\}$.
Here both hyperedges are transversal. Let us choose one of them, say $H=\{2,3\}$.
With this we can consider the $H$-move in $NIM_\cH$ that decreases both piles $2$ and $3$ to the minimum of the corresponding positions:
\[
a=(1,2,4,2) ~\to~ a'=(1,1,1,2)  ~~\text{ and }~~ b=(4,5,7,5) ~\to~ b'=(4,4,4,5).
\]
It is easy to see now that both $a'$ and $b'$ are $\P$ positions.

Similar examples can be constructed for e.g., the other JM hypergraphs listed in sections \ref{s4}, for example for the six graphs that are JM. Note that the height function for those may be much more complicated.

\medskip

The structure of the paper is as follows. In Section \ref{s2}, we fix our notation and define basic concepts.
We also prove several properties of JM hypergraphs that are needed for our proofs later.
In Section \ref{sec-necessary}, we show necessary conditions for a hypergraph to be JM.
In Section \ref{s3} we provide a general sufficient condition for a hypergraph to be JM.
In Section \ref{s4} we apply the general sufficient condition, and show that several other families of hypergraphs are JM. Among them, we provide a complete characterization of 
JM graphs.
In Section \ref{s-size}, we discuss the size of $k$-uniform JM hypergraphs.
Finally, in Section \ref{sec-cr}, we present further examples of JM hypergraphs and discuss related topics.

\section{Basic Concepts and Notation}
\label{s2}

In this section we introduce the basic notation and definitions.
We prove some of the basic properties of $NIM_\cH$ games, the height, and the JM formula.

\subsection{$NIM_\cH$ Games}
\label{ss21}

We need to recall first the precise definition of impartial games and the SG function.

To a subset $S\subseteq \ZZP$ of nonnegative integers let us associate its \emph{minimal excludant}
$\mex(S)=\min\{i\in \ZZP\mid i\not\in S\}$, that is,
the smallest nonnegative integer that is not included in $S$. Note that $\mex(\emptyset)=0$.

An impartial game $\Gamma$
is played by two players over a (possibly infinite) set $X$ of positions.
They take turns to move, and the one who cannot move is the loser.
For a position $x\in X$ let us denote by $N^+(x)\subseteq X$
the set of positions $y\in N^+(x)$ that are reachable from $x$ by a single move.
For $y\in N^+(x)$ we denote by $x\to y$ such a move. We assume that the same set of moves are available for both players from every position.
We also assume that no matter how the players play and which position they start, the game ends in a finite number of moves.
The SG function $\G_\Gamma$ of the game is a mapping $\G_\Gamma:X\mapsto \ZZP$ that associates a nonnegative integer to every position, defined by the following  recursive formula:
\[
\G_\Gamma(x) ~=~ \mex\{\G_\Gamma(y)\mid y\in X, ~s.t.~ \exists x\to y\} .
\]

In our proofs we shall use the following, more combinatorial characterization of SG functions that can be derived easily from the above definition.
Assume that $\Gamma$ is an impartial game over the set of positions $X$, and
$g:X\to \ZZP$ is a given function.
Then, $g$ is the SG function of $\Gamma$ if and only if the following two conditions hold:
\begin{itemize}
\item[\rm (A)] For all positions $x\in X$ and moves $x\to y$ in $\Gamma$ we have $g(x)\neq g(y)$.
\item[\rm (B)] For all positions $x\in X$ and integers $0\leq z<g(x)$ there exists a move $x\to y$ in $\Gamma$ such that $g(y)=z$.
\end{itemize}

It is easy now to verify that for a hypergraph $\cH\subseteq 2^V$ the game $NIM_\cH$ is indeed an impartial game over the infinite set $X=\ZZP^V$ of positions.

\medskip

Let us note that all quantities used in
\eqref{e-m}--\eqref{e-JM-II}, $m(x)$, $y_\cH(x)$, $v_\cH(x)$, $h_\cH(x)$ as well as $\U_\cH$ are well defined
for an arbitrary hypergraph $\cH$.  Let us also note that the values $m(x)$ and $y_\cH(x)$ determine completely the value of $v_\cH(x)$.

\medskip

Following \cite{JM80} we call a position $x\in \ZZP^V$
\emph{long} if $m(x) \leq \binom{y^{}_\cH(x)}{2}$
(that is, if  $\U_\cH(x) = h_\cH(x)$) and call it
\emph{short} if $m(x) > \binom{y^{}_\cH(x)}{2}$ (that is, if $\U_\cH(x) = v_\cH(x)$).

According to the rules of $NIM_\cH$, if $x\to x'$ is a move then for the set $H=\{i\mid i\in V,~ x_i>x'_i\}$ we must have $H\in \cH$. We call such a move an $H$-\emph{move}.
For a subset $S\subseteq V$ let us denote by $\chi(S)$ the characteristic vector of $S$, that is, $\chi(S)_j=1$ if $j\in S$ and $\chi(S)_j=0$ if $j\not\in S$.
We denote for a position $x\in \ZZP^V$ and hyperedge $H\in\cH$ by $x^{s(H)}$ the vector $x-\chi(H)$.
Thus, $x\to x^{s(H)}$ is an $H$-move in $NIM_\cH$. We call such a move also a \emph{slow move}, since we have $x'\leq x^{s(H)}$ for all $H$-moves $x\to x'$. Let us add that for vectors we use relations $\leq$, $<$ and $=$ componentwise, as usual.

\subsection{Height}
\label{ss22}

Let us recall that the height $h(x)$ of a position $x$ is defined as the maximum number of consecutive moves players can make starting with $x$.
Any such longest sequence of moves will be called a \emph{height sequence} and any move in it is called a \emph{height move}. Let us note that in any height sequence each move can be replaced by the corresponding slow move. Furthermore, such moves can be executed in any order. In particular, if a height sequence starting in position $x$ involves an $H$-move, then $h(x^{s(H)})=h(x)-1$.

Let us next observe some basic properties of the height that will be instrumental in our proofs.
We assume in the sequel that a hypergraph $\cH\subseteq 2^V$ is fixed, and all positions mentioned are from $\ZZP^V$.

\begin{lemma}\label{l0}
For every position $x\in\ZZP^V$ we have
\[
\G_\cH(x)\leq h_\cH(x).
\]
\end{lemma}

\proof
By the definition of the SG function, for every position $x$ with $\G_\cH(x)>0$ we have a move $x\to x'$ such that $\G_\cH(x')=\G_\cH(x)-1$.
Thus, starting from $x$ we can make $\G_\cH(x)$ consecutive moves
(in each move decreasing the SG function exactly by $1$.) On the other hand,
the height is the maximum number of such consecutive moves, proving thus the above inequality.
\qed

\begin{lemma}\label{l1}
If $x\geq x'$ are two positions, then
\[
h_\cH(x)~\geq~ h_\cH(x') ~\geq~ h_\cH(x) - \left(\sum_{i=1}^n (x_i-x'_i)\right).
\]
In particular, if $x'$ differs from $x$ only in one of its coordinates and only by one unit, then $h_\cH(x')\geq h_\cH(x)-1$.
\end{lemma}

\proof
Any sequence of moves starting with $x'$ can be repeated from $x$, and hence $h_\cH(x)~\geq~ h_\cH(x')$.
Furthermore decreasing one of the piles by one unit can decrease the height by at most one.
From this the second inequality follows.
\qed

\begin{corollary}\label{c0}
If $x,x'\in\ZZP^V$ are two positions such that for some index $j$ we have
\[
x'_j=x_j-1  \text{ and }   x'_i\geq x_i \text{ for all } i\neq j,
\]
then $h_\cH(x')\geq h_\cH(x)-1$.
\end{corollary}

\proof
Follows from Lemma \ref{l1}.
\qed

\begin{lemma}\label{l2}
Assume $x\to x'$ is an $H$-move for some $H\in\cH$. Then for every integer $z$ such that
$h_\cH(x')\leq z \leq h_\cH(x^{s(H)})$ there exists an $H$-move
$x\to x''$ for which $h_\cH(x'')=z$ and $x'\leq x''\leq x^{s(H)}$.
\end{lemma}

\proof
Consider a sequence of positions $x^0=x^{s(H)}\geq x^1 \geq \dots  \geq x^p=x'$,
where $\sum_{j=1}^n (x_{j}^{i-1}-x_{j}^{i})=1$ for all $i=1,\dots ,p$.
By Lemma \ref{l1} we have $h_\cH(x^{i-1})\geq h_\cH(x^i) \geq h_\cH(x^{i-1})-1$ implying that
for every integer $h_\cH(x^p)=h_\cH(x')\leq z \leq h_\cH(x^{s(H)})=h_\cH(x^0)$
there exists an index $0\leq i \leq p$ such that $h_\cH(x^i)=z$.
Note also that by the definition of these vectors, we have for all $i=0,...,p$ the relations
$x^i_j\leq x^0_j<x_j$ for all $j\in H$ and $x_j=x^p_j\leq x^i_j\leq x^0_j=x_j$ for all $j\not\in H$.
Thus, $x^i\leq x$ and $\{j\mid x^i_j<x_j\}=H$. Thus all of these positions are reachable from $x$ by an $H$-move.
\qed

To illustrate the above we consider the same hypergraph
$\cH=\{\{1,2\},\{1,4\},\{2,3\},\{3,4\}\}$ as in the introduction,
and the positions $x=(4,5,2,3)$ and $x'=(1,1,2,3)$.
Then  $x\to x'$ is an $H$-move with $H=\{1,2\}$.
With this hyperedge we have $x^{s(H)}=(3,4,2,3)$.
Thus, a possible sequence illustrating the proof of the lemma could be
$x^0=(3,4,2,3)$, $x^1=(2,4,2,3)$, $x^2=(2,3,2,3)$, $x^3=(2,2,2,3)$, $x^4=(1,2,2,3)$ and $x^5=x'=(1,1,2,3)$.
Recall that for this hypergraph the height function is given by the expression
$h_\cH(x_1,x_2,x_3,x_4)=\min\{x_1+x_3,x_2+x_4\}$.
Thus, for the above sequence we obtain the height values 
$5$, $4$, $4$, $4$, $3$, $3$, which form an interval $[3,5]$ as claimed.

\medskip

\begin{lemma}\label{l2-y}
Consider an arbitrary position $x\in\ZZP^V$ and hyperedge $H\in\cH$ such that $h_\cH(x^{s(H)})=h_\cH(x)-1$ and $m(x^{s(H)})=m(x)-1$.
Then we have $y_\cH(x^{s(H)})\geq y_\cH(x)$.
\end{lemma}

\proof
Since $m(x^{s(H)})=m(x)-1$ we have the inequality $x^{s(H)}-m(x^{s(H)})e \geq x-m(x)e$ and thus the claim follows from Lemma \ref{l1}.
\qed

\subsection{JM Formula}
\label{ss23}

For a positive integer $\eta\in\ZZP$ let us associate the set
\begin{equation}\label{e-Z}
Z(\eta) ~=~ \left\{ i\in\ZZP \left| \binom{\eta}{2} \leq i < \binom{\eta +1}{2} \right.\right\}.
\end{equation}
It is immediate to see the following properties:
\begin{lemma}\label{l-Z}
If $\eta\neq \eta'$ then $Z(\eta)\cap Z(\eta')=\emptyset$. Furthermore, we have
\[
\ZZP ~=~ \bigcup_{\eta=1}^{\infty} Z(\eta).
\]
\qed
\end{lemma}

Let us recall next that by the definitions of the quantities in \eqref{e-myv} the value $v_\cH(x)$ depends only on the pair of integers $m(x)$ and $y_\cH(x)$.

\begin{lemma}\label{l-v}
For an arbitrary positive integer $\eta\in\ZZP$ we have
\[
\left\{ v_\cH(x) ~\left|~ x\in\ZZP^V, ~ y_\cH(x)=\eta\right.\right\} ~=~ Z(\eta).
\]
\end{lemma}

\proof
Follows by \eqref{e-v} and the fact that $m(x)$ can take arbitrary integer values modulo $y_\cH(x)=\eta$.
\qed

\begin{lemma}\label{l-my}
For an arbitrary position $x\in\ZZP^V$ and move $x\to x'$ in $NIM_\cH$ we have $(m(x),y_\cH(x))\neq (m(x'),y_\cH(x'))$.
\end{lemma}

\proof
If $m(x)=m(x')$ and $x\to x'$ is an $H$-move for a hyperedge $H\in \cH$,
then we have the inequality $x-\chi(H)\geq x'$, where $\chi(H)$ is the characteristic vector of $H$.
This implies
\[
x-m(x)e \geq \chi(H) + x'-m(x')e
\]
from which $y_\cH(x)\geq y_\cH(x')+1$ follows by \eqref{e-y}.
\qed

\begin{lemma}\label{l4}
A position $x \in \ZZP^V$ is 
long if and only if $v_\cH(x)\geq m(x)$.
\end{lemma}

\proof
Note first that if $x$ is long 
then $m(x)\leq \binom{y_\cH(x)}{2}$ by \eqref{e-v}, and thus,
by Lemma \ref{l-v} it follows that $m(x)\leq v_\cH(x)$.
On the other hand, if $x$ is short 
then we have $m(x)>\binom{y_\cH(x)}{2}$ by \eqref{e-JM-II}, and thus,  $m(x)-\binom{y_\cH(x)}{2}-1\geq 0$, implying
\[
m(x)-\binom{y_\cH(x)}{2}-1 \geq \left(\left(m(x)-\binom{y_\cH(x)}{2}-1\right) \mod y_\cH(x) \right)
\]
from which by \eqref{e-v} it follows that $m(x)-1\geq v_\cH(x)$.
\qed

\begin{lemma}\label{l5}
For an arbitrary position $x\in\ZZP^V$ and move $x\to x'$ such that $m(x)\leq \U_\cH(x')$ position $x'$ is long. 
\end{lemma}

\proof
If $x'$ were short 
then by Lemma \ref{l4} we would get $\U_\cH(x')=v_\cH(x')<m(x')\leq m(x)$, contradicting $m(x)\leq \U_\cH(x')$.
\qed



\begin{lemma}\label{l-H1H2-my}
Let $\cH,\hH\subseteq 2^V$ be two hypergraphs such that $\hH$ contains a hyperedge different from $V$.
Then for every position $x\in\ZZP^V$ there exists a position $\hx\in\ZZP^V$ such that $m(x)=m(\hx)$ and $y_\cH(x)=y_\hH(\hx)$.
\end{lemma}

\proof
Choose a minimal hyperedge $H\in \hH$, and consider the position $\hx=m(x)e + (y_\cH(x)-1)\chi(H)$. Since $\hH$ is assumed to have a hyperedge different from $V$, we can choose $H$ such that $H\neq V$.
Therefore, we have $m(\hx)=m(x)$, and $\hx-m(\hx)e=(y_\cH(x)-1)\chi(H)$, implying $y_\cH(x)=y_\hH(\hx)$, since $H$ was chosen as a minimal hyperedge, and thus from position $\hx-m(\hx)e$ we can have only $H$-moves.
\qed

Let us also note that if $\cH\subseteq \hH\subseteq 2^V$ be two (nested) hypergraphs,
then for every position $x\in\ZZP^V$ we have $h_\cH(x)\leq h_\hH(x)$, by the definition of height.

\section{Necessary Conditions}
\label{sec-necessary}

%
In this section we prove some properties of JM hypergraphs.

For every hypergraph $\cH\subseteq 2^V$, we assume that $V=\bigcup_{H \in \cH} H$.

Let us recall that a hypergraph $\cH \subseteq 2^V$ is \textit{not connected}
if $V$ can be partitioned into two nonempty subsets $V_1$ and $V_2$ such that every hyperedge of $\cH$
is contained in either $V_1$ or $V_2$. Otherwise $\cH$ is called \textit{connected}.

\begin{lemma}\label{necessaryconditions}
A JM hypergraph $\cH$ is connected.
\end{lemma}

\proof
Let us assume indirectly that $\cH$ is not connected. Let $H_1,H_2 \in \cH$ be two minimal hyperedges in two different connected components of $\cH$. Consider the position $x$ defined as follows:
$$x_i= \begin{cases}
1 & \textrm{if }i \in H_1 \\
3 & \textrm{if }i \in H_2\\
0 & \textrm{otherwise.}
\end{cases}$$
This position has $\G_{\cH}(x)=2$ (the $NIM$ sum of $1$ and $3$). It has $m(x)\in \{0,1\}$ and correspondingly $y^{}_{\cH}(x)\in\{3,5\}$.
Therefore, $x$ is long.  
Since $h_{\cH}(x)=4$ we can conclude $\U_{\cH}(x) \neq \G_{\cH}(x)$, which contradicts the assumption.
\qed

\begin{lemma}\label{l3}
If $\cH$ is a JM hypergraph, then it is minimal transversal-free.
\end{lemma}

\proof
Let us assume first indirectly that $H_0 \in \cH$ is a transversal of $\cH$, say $H_0\in\cH$ intersects all hyperedges of $\cH$.
Consider the position $x\in\ZZP^V$ defined by $x_i=1$, $i\in V$. For this position we have $m(x)=1$, $y_\cH(x)=1$, and $v_\cH(x)=0$.
Thus, $x$ is short 
and $\U_\cH(x)=0$. On the other hand, a slow $H_0$-move $x\to x'$ takes us into a position with $x'_i=0$, $i\in H_0$.
Since $H_0$ intersects all hyperedges of $\cH$ we must have $h_\cH(x')=0$. Thus by Lemma \ref{l0} we have $\G_\cH(x')=0$, which implies by property (A) of the SG function that $\G_\cH(x)\neq 0$,
or in other words that $\G_\cH(x)\neq \U_\cH(x)$, which implies on its turn that $\cH$ is not JM.
This contradiction implies that $\cH$ is transversal-free.

To see minimality with respect to the transversal-freeness, let us consider an arbitrary proper subset $S\subset V$
for which the induced subhypergraph $\cH_S$ is not empty, and a position with $x_i=0$ for all $i\in V\setminus S$, and $x_i>0$ for all $i\in S$.
Every move from $x$ is an $H$-move for some $H\in\cH_S$. Furthermore, $h_\cH(x)>0$, $m(x)=0$ (since $V\setminus S\neq\emptyset$), and thus,
all such positions are long,  
implying (by our assumption that $\cH$ is JM) that
$\G_\cH(x)=\U_\cH(x)=h_\cH(x)>0$ for all such positions.
Thus, we must have a move $x\to x'$ such that $\G_\cH(x')=h_\cH(x')=0$.
This is possible only if this move is an $H$-move for a hyperedge $H\in\cH_S$ that intersects all hyperedges of $\cH$.
\qed


In the rest of this section, we study further properties of transversal-free hypergraphs.
This is used to obtain sufficient conditions for a hypergraph to be JM.

\begin{lemma}\label{l3T}
If $\cH$ is transversal-free and $x\to x'$ is a move in $NIM_\cH$, then we have $h_\cH(x')\geq m(x)$.
\end{lemma}
\proof
Since $\cH$ is transversal-free,  for every hyperedge $H\in\cH$
there exists an $H'\in\cH$ such that $H\cap H'=\emptyset$. Consequently, for every move $x\to x'$ we must have $h_\cH(x')\geq m(x)$.
This is because if $x\to x'$ is an $H$-move and $H'\in\cH$ is disjoint from $H$, then we have $x'_i=x_i\geq m(x)$ for all $i\in H'$.
Thus we can make at least $m(x)$ slow $H'$-moves, implying the claim.
\qed

\begin{lemma}\label{l6}
If $\cH$ is a transversal-free hypergraph, $x\in\ZZP^V$ is a long 
position, and $x \to x'$ is a move such that $0\leq \U_\cH(x')<m(x)$, then
$x'$ is a short 
position, for which we have $m(x')\leq m(x)$ and $m(x)-m(x')+1\leq y_\cH(x')<y_\cH(x)$.
\end{lemma}

\proof
By Lemma \ref{l3T} we have $h_\cH(x')\geq m(x)$.
Now, if $x'$ were long 
then $\U_\cH(x')=h_\cH(x')\geq m(x)$ would follow, contradicting our assumption that $\U_\cH(x')<m(x)$.
The inequality $m(x')\leq m(x)$ holds for any move $x\to x'$. Since $x$ is long 
and  $x'$  is short, 
we have the inequalities
\[
\binom{y_\cH(x)}{2}~\geq~ m(x)~\geq~ m(x') ~>~ \binom{y_\cH(x')}{2}
\]
from which $y_\cH(x')<y_\cH(x)$ follows.
Assume next that $x\to x'$ is an $H$-move for some hyperedge $H\in\cH$.
Since $\cH$ is transversal-free, there exists $H'\in\cH$ such that $H\cap H'=\emptyset$.
Then we have $x'_i=x_i\geq m(x)$ for all $i\in H'$, implying that $x_i'-m(x')\geq m(x)-m(x')$ for all $i\in H'$. Thus, from position $x'-m(x')e$ we can make at least $m(x)-m(x')$ slow $H'$-moves,
and thus $y_\cH(x')\geq m(x)-m(x')+1$ follows by \eqref{e-y}.
\qed

\begin{lemma}\label{l7}
If $\cH$ is a transversal-free hypergraph, $x\in\ZZP^V$ is a short 
position, and $x\to x'$ is a move such that $0 \leq \U_\cH(x')<\U_\cH(x)$, then
$x'$ must also be a short 
position, for which we have $m(x')\leq m(x)$ and $m(x)-m(x')+1\leq y_\cH(x')\leq y_\cH(x)$ with $(m(x),y_\cH(x))\neq (m(x'),y_\cH(x'))$.
\end{lemma}

\proof
Since $x$ is short, 
we have $\U_\cH(x)=v_\cH(x)<m(x)$ by Lemma \ref{l4}. We also have $h_\cH(x')\geq m(x)$ by Lemma \ref{l3T}. Thus,
$\U_\cH(x')<\U_\cH(x)=v_\cH(x)<m(x)\leq h_\cH(x')$ follows by our assumption, implying $\U_\cH(x')\neq h_\cH(x')$.
Thus, $x'$ is not long. 
The inequality $m(x')\leq m(x)$ holds for any move $x\to x'$.
Lemma \ref{l-v} and $v_\cH(x')<v_\cH(x)$ implies $y_\cH(x')\leq y_\cH(x)$.
Furthermore, $v_\cH(x')< v_\cH(x)$ also implies $(m(x),y_\cH(x))\neq (m(x'),y_\cH(x'))$ since the pair $(m(x),y_\cH(x))$ determines $v_\cH(x)$ uniquely.
Assume next that $x\to x'$ is an $H$-move for some hyperedge $H\in\cH$. Since $\cH$ is transversal-free, there exists $H'\in\cH$ such that $H\cap H'=\emptyset$.
Then we have $x'_i=x_i\geq m(x)$ for all $i\in H'$, and thus,
$y_\cH(x')\geq m(x)-m(x')+1$ follows by \eqref{e-y}.
\qed

\begin{lemma}\label{l8}
If a hypergraph $\cH\subseteq 2^V$ is transversal-free, then the function $\U_\cH$ satisfies property (A),
that is, for all moves $x\to x'$ in $NIM_\cH$ we have $\U_\cH(x)\neq \U_\cH(x')$.
\end{lemma}

\proof
To prove this statement, we consider four cases, depending on the types of the positions $x$ and $x'$, which can be long or short.

If both $x$ and $x'$ are long 
then $\U_\cH(x) = h_\cH(x)\neq h_\cH(x')=\U_\cH(x')$,
since every move strictly decreases the height by its definition.

If $x$ is long 
and $x'$ is short 
then we have $\U_\cH(x)=h_\cH(x) > m(x)\geq m(x') > v_\cH(x')=\U_\cH(x')$, proving the claim.
Here the first strict inequality is implied by the fact that every move strictly decreases the height and
by Lemma \ref{l3T} yielding $h_\cH(x) > h_\cH(x') \geq m(x)$.
The inequality $m(x)\geq m(x')$ holds for every move $x\to x'$. Finally $m(x')>v_\cH(x')$ is implied by Lemma \ref{l4}.

If $x$ is short 
and $x'$ is long 
then we have $\U_\cH(x')=h_\cH(x')\geq m(x)$ by Lemma \ref{l3T}, and
$m(x) > v_\cH(x)=\U_\cH(x)$ by Lemma \ref{l4}, which together imply the claim.

Finally, if both $x$ and $x'$ are short 
then we have $(m(x),y_\cH(x))\neq (m(x'),y_\cH(x'))$ by Lemma \ref{l-my}.
If $y_\cH(x)\neq y_\cH(x')$, then $v_\cH(x)\neq v_\cH(x')$ follows by Lemma \ref{l-v},
since in this case $Z(y_\cH(x))\cap Z(y_\cH(x'))=\emptyset$.
If $y_\cH(x)=y_\cH(x')$ then by \eqref{e-v} we have $v_\cH(x)=v_\cH(x')$ if and only if
$m(x')=m(x)-\alpha y_\cH(x)$ for some positive integer $\alpha$.
Thus, we must have $m(x')\leq m(x)-y_\cH(x)$.
This implies, by \eqref{e-y}, that $y_\cH(x')\geq y_\cH(x)+1$, which contradicts $y_\cH(x)=y_\cH(x')$, completing the proof of our statement.
To see the last implication, recall that $\cH$ is transversal-free.
Thus, if $x\to x'$ is an $H$-move, then there is a hyperedge $H'\in\cH$ such that $H\cap H'=\emptyset$.
For this hyperedge we have $x'_i=x_i\geq m(x)$ for all $i\in H'$, from which the claim follows by \eqref{e-y}.
\qed

The above lemma has the following consequence.

\begin{theorem}\label{t-addaset}
 Let $\cH$ be a JM hypergraph, $H,H'\in\cH$, $H \cap H'=\emptyset$, and $H \subseteq S \subseteq V\setminus H'$.
 Then  $\cH^+=\cH\cup\{S\}$ is also a JM hypergraph with $\G_{\cH^+}=\G_\cH$.
\end{theorem}

\proof
Let note first that $H\subseteq S$ implies that $h_{\cH^+} = h_{\cH}$ and consequently
$y_{\cH^+}=y_{\cH}$, $v_{\cH^+}=v_{\cH}$, and thus $\U_{\cH^+}=\U_{\cH}$.
Furthermore, any move in $NIM_{\cH}$ is still a move in $NIM_{\cH^+}$.
Therefore, for every $0\leq v<\U_{\cH^+}(x)$ there exists a move $x\to x'$ in $NIM_{\cH^+}$ such that $\U_{\cH^+}(x')=v$.
Finally, by $S\cap H'=\emptyset$ the hypergraph $\cH^+$ is also transversal-free, and thus by Lemma \ref{l8} we have $\U_{\cH^+}(x')\neq\U_{\cH^+}(x)$ for all moves $x\to x'$ of $NIM_{\cH^+}$.
\qed

\section{Sufficient Conditions}
\label{s3}

Let us first recall that properties (A) and (B) characterize the SG function of an impartial game. We can reformulate these now for $NIM_\cH$, and obtain the following necessary and sufficient condition for $\cH$ to be JM:

\begin{lemma}\label{l-ns-JM}
A hypergraph $\cH\subseteq 2^V$ is JM if and only if the following conditions hold:
\begin{itemize}
\item[\rm (A0)] $\cH$ is transversal-free.
\item[\rm (B1)] For every long 
position $x\in\ZZP^V$ and integer $m(x)\leq z < h_\cH(x)$ there exists a move $x\to x'$ such that $x'$ is long 
and $h_\cH(x')=z$.
\item[\rm (B2)] For every long 
position $x\in\ZZP^V$ and integer $0\leq z < m(x)$ there exists a move $x\to x'$ such that $x'$ is short 
and  $v_\cH(x')=z$.
\item[\rm (B3)] For every short 
position $x \in \ZZP^V$ and integer $0 \leq z < v_\cH(x)$ there exists a move $x \to x'$ such that $x'$ is short 
and $v_\cH(x')=z$.
\end{itemize}
\end{lemma}

\proof
It is easy to see by Lemmas \ref{l5}, \ref{l6} and \ref{l7} that
conditions (B1), (B2), and (B3) are simple and straightforward reformulations
of condition (B) for the case of $NIM_\cH$ and the function $g = \U_\cH$ defined in \eqref{e-JM-I}-\eqref{e-JM-II}.

For the reverse direction, Lemma \ref{l8} shows that condition (A0) implies (A),
while Lemma \ref{l3} shows that if $\cH$ is JM, then it is also transversal-free. Furthermore, if $\cH$ is JM, that is if $\G_\cH=\Uparrow_\cH$, then  Lemmas \ref{l4} and \ref{l5} and the properties of an SG function imply (B1), (B2) and (B3).
\qed

\subsection{General Sufficient Conditions}

Let us next replace conditions (B2) and (B3) with somewhat simpler sufficient conditions.

\begin{lemma}\label{l-suff-23}
If a hypergraph $\cH\subseteq 2^V$ satisfies the following two conditions then it also satisfies {\rm (B2)} and {\rm (B3)}.
\begin{itemize}
\item[\rm (C2)] For every position $x\in\ZZP^V$ and integer $1\leq \eta < y_\cH(x)$ there exists a move $x\to x'$ such that $m(x')=m(x)$ and $y_\cH(x')=\eta$.
\item[\rm (C3)] For every position $x\in\ZZP^V$ and integers $0\leq \mu < m(x)$ and $m(x)-\mu +1 \leq \eta \leq y_\cH(x)$ there exists a move $x\to x'$ such that $m(x')=\mu$ and $y_\cH(x')=\eta$.
\end{itemize}
\end{lemma}

\proof
Let us consider first another hypergraph $\hH=\{S\subseteq V\mid 1\leq |S|\leq n-1\}$. Then by the earlier cited result of Jenkyns and Mayberry \cite{JM80}
$\hH$ is a JM hypergraph. Note also that the games $NIM_\cH$ and $NIM_\hH$ are both played over the same set of positions $\ZZP^V$.

Let us now consider a position $x\in\ZZP^V$. By Lemma \ref{l-H1H2-my} there exists a position $\hx\in\ZZP^V$ such that $m(x)=m(\hx)$ and $y_\cH(x)=y_\hH(\hx)$.
Now, let us observe that since $\hH$ is JM, properties (B2) and (B3) are satisfied by Lemma \ref{l-ns-JM}.
Let us also note that if $\hx\to \hx'$ is a move in $NIM_\hH$ guaranteed to exist by properties
(B2) and (B3) then Lemmas \ref{l6} and \ref{l7} show that $(m(\hx'),y_\hH(\hx')) ~\in~ S(m(\hx),y_\hH(\hx))=S(m(x),y_\cH(x))$, where the set $S(\alpha,\beta)$ is defined as
\[
S(\alpha,\beta) ~=~ \left\{ (\mu,\eta) \left| \begin{array}{c}0\leq \mu \leq \alpha,\\ \alpha-\mu +1\leq \eta \leq \beta \end{array}\right.\right\} \setminus \{(\alpha,\beta)\}.
\]
Let us observe next that properties (C2) and (C3) imply that for every $(\mu,\eta)\in S(m(x),y_\cH(x))$ there exists a move $x\to x'$ in $NIM_\cH$ such that $m(x')=\mu$ and $y_\cH(x')=\eta$.
Thus, for every move $\hx\to\hx'$ in $NIM_\hH$ that validates properties (B2) and (B3) for $\hH$ we have a corresponding move $x\to x'$ in
$NIM_\cH$ such that $m(\hx')=m(x')$ and $y_\cH(x')=y_\hH(\hx')$, implying by \eqref{e-v}
that $v_\cH(x')=v_\hH(\hx')$ and that $x'$ and $\hx'$ are of the same type, that is, both are long or both are short by \eqref{e-JM-I} and \eqref{e-JM-II}.

Consequently, properties (C2) and (C3) do imply properties (B2) and (B3), as claimed.
\qed

\begin{corollary}\label{c1}
If a hypergraph $\cH$ satisfies properties {\rm (A0)}, {\rm (B1)}, {\rm (C2)}, and {\rm (C3)}, then it is JM.
\end{corollary}

\proof
The claim follows by Lemmas \ref{l-suff-23} and \ref{l-ns-JM}.
\qed

\subsection{Simplified Sufficient Conditions}

The conditions in Corollary \ref{c1} still involve the existence of moves with certain properties. In this section we further weaken those conditions, and replace them with easier to check properties of the hypergraph itself.

Given a hypergraph $\cH\subseteq 2^V$, we say that a sequence of hyperedges $H_0,H_1,\dots ,H_p$ of $\cH$ is a \emph{chain} if
\begin{equation}\label{e-chain}
H_{k+1}\cap H_k\neq\emptyset ~~~\text{ and }~~~ |H_{k+1}\setminus H_k|\leq 1 ~~ \mbox{for all } k=0,\dots ,p-1.
\end{equation}
For convenience, $p=0$ is possible, that is, a single set is considered to be a chain.

For a subfamily $\cF\subseteq \cH\subseteq 2^V$ we denote by $V(\cF)$ the set of vertices of $\cF$, that is, $V(\cF)=\bigcup_{F\in\cF}F$. In particular we have $V(\cH)=V$.

We shall consider the following properties:
\begin{itemize}
\item[\rm (A1)] $\cH$ is minimal transversal-free.
\item[\rm (D1)] For every pair of hyperedges $H,H'\in\cH$ there exists a chain \linebreak $H_0,H_1,\dots ,H_p$ in $\cH$ such that $H=H_0$, $H'=H_p$, and $H_i\subseteq H\cup H'$ for $i=1,\dots,p$.
\item[\rm (D2)] For every position $x\in\ZZP$ with $m(x)>0$ there exists a hyperedge $H\in\cH$ such that
$h_\cH(x^{s(H)})= h_\cH(x)-1$ and $m(x^{s(H)})=m(x)-1$.
\end{itemize}

Our main claim in this section is that the above properties are sufficient for a hypergraph to be JM.

\begin{theorem}\label{t2}
If a hypergraph $\cH$ satisfies properties {\rm (A1)}, {\rm (D1)}, and {\rm (D2)}, then it is JM.
\end{theorem}

Let us first remark that condition (D2) may be necessary for a hypergraph to be JM, but we cannot prove this.

To prove Theorem \ref{t2}, we show several consequences of the above properties.
In particular, we show that properties (A1), (D1), and (D2) imply properties (A0), (B1), (C2) and (C3), and thus,
Theorem \ref{t2} follows by Corollary \ref{c1}.

Clearly (A1) implies (A0).  The other three implications we will show separately.

\begin{lemma}\label{l11}
If a hypergraph $\cH\subseteq 2^V$ satisfies properties {\rm (A1)}, {\rm (D1)}, and {\rm (D2)},
then it also satisfies property (B1), that is, for all long positions $x \in \ZZP^V$ 
and for all integer values $m(x)\leq z < h_\cH(x)$ there exists a move $x \to y$ such that $y$ is long 
and $h_\cH(y)=z$.
\end{lemma}

\proof
Let us first consider a long position $x\in\ZZP^V$ with $m(x)>0$. By property (D2) there exists a $j\in H\in\cH$ such that $h_\cH(x^{s(H)})=h_\cH(x)-1$ and $x_j=m(x)$. By property (A1) the subhypergraph $\cH_{V\setminus\{j\}}$ contains a transversal $H'\in\cH$.
By property (D1) we have then a chain $H_0,H_1,\dots ,H_p$ such that $H_0=H$, $H_p=H'$ and $|H_{k+1}\setminus H_k|\leq 1$ for all $k=0,1,\dots ,p-1$. Let us then define positions $x^{\ga,k}$ and $x^{\go,k}$ for $k=1,\dots ,p$ by
\[
x^{\ga,k}_i~=~\begin{cases} 0&\text{ if } i\in H_{k-1}\cap H_k\\x_i-1&\text{ if } i\in H_k\setminus H_{k-1}\\x_i &\text{ if } i\not\in H_k,\end{cases}
~~~
x^{\go,k}_i~=~\begin{cases} 0&\text{ if } i\in H_k\\x_i &\text{ if } i\not\in H_k.\end{cases}
\]
Set $x^{\ga,0}=x^{s(H)}$, and define $x^{\mu,0}$ and $x^{\go,0}$ by
\[
x^{\mu,0}_i~=~\begin{cases} 0&\text{ if } i=j\\x_i-1&\text{ if } i\in H_0\setminus\{j\}\\
x_i &\text{ if } i\not\in H_0,\end{cases}
~~~
x^{\go,0}_i~=~\begin{cases} 0&\text{ if } i\in H_0\\x_i &\text{ if } i\not\in H_k.\end{cases}
\]

We claim first that all positions $x^{\ga,k}\geq y\geq x^{\go,k}$ are long 
and are reachable from $x$ by an $H_k$-move $x \to y$, for $k=1,\dots ,p$.
This is because $m(y)=0$ for all these positions since they have a zero
(namely $y_i=0$ for all $i\in  H_{k-1}\cap H_k$, which is not an empty set by \eqref{e-chain}).
The analogous claim holds for positions $x^{\mu,0}\geq y\geq x^{\go,0}$ since $y_j=0$ for all these positions.
We also claim that positions $x^{\ga,0}\geq y\geq x^{\mu,0}$ are also long 
whenever  $x$ is long 
(and they are reachable from $x$ by an $H_0$-move $x \to y$).
The last claim is true because we have $m(y)\leq m(x^{\ga,0})=m(x)-1$, and $y_\cH(y)\geq y_\cH(x^{\ga,0})\geq y_\cH(x)$ by Lemma \ref{l2-y},
and thus, the fact that $x$ is long 
implies $m(y) < m(x)\leq \binom{y_\cH(x)}{2}\leq \binom{y_\cH(y)}{2}$.

Let us observe next that the sets of height values for these ranges of positions form intervals by Lemma \ref{l2}. Namely, we have
\[
\begin{array}{cll}
\left\{h_\cH(y)\mid x^{\ga,0}\geq y\geq x^{\mu,0} \right\} &=& \left[ h_\cH(x^{\mu,0}), h_\cH(x)-1 \right], \\*[2mm]
\left\{h_\cH(y)\mid x^{\mu,0}\geq y\geq x^{\go,0} \right\} &=& \left[ h_\cH(x^{\go,0}), h_\cH(x^{\mu,0}) \right], \text{ and } \\*[2mm]
\left\{h_\cH(y)\mid x^{\ga,k}\geq y\geq x^{\go,k} \right\} &=& \left[ h_\cH(x^{\go,k}), h_\cH(x^{\ga,k}) \right] \text{ for } k=1,\dots ,p.
\end{array}
\]
We claim that these intervals cover all values in the interval $[m(x),h_\cH(x)-1]$,
as stated in the lemma. To see this claim, we show the following inequalities:
\[
\begin{array}{lll}
h_\cH(x^{\ga,k}) &\geq& h_\cH(x^{\go,k-1})-1 ~~\text{ for } 1\leq k \leq p, \text{ and }\\
h_\cH(x^{\go,p}) &\leq& m(x).
\end{array}
\]
The first group of inequalities follow by Corollary \ref{c0}. This is because for $i\not\in H_k\setminus H_{k-1}$ we have $x^{\ga,k}_i\geq x^{\go,k-1}_i$ and for the unique index $i\in H_k\setminus H_{k-1}$ we have $x^{\ga,k}_i=x_i-1$ and $x^{\go,k-1}_i=x_i$.

For the second inequality  observe that by our choice the set $H'=H_p$ intersects every hyperedge that does not contain $j\in V$.
Thus, the only possible moves from $x^{\go,p}$ are $H$-moves for hyperedges
$H\in\cH$ that contain element $j$. Since $x_j=m(x)$, the total number of such moves is limited by $m(x)$, as stated.

\medskip

Let us next consider a long position $x\in\ZZP^V$ such that $m(x)=0$. Consider $W=\{i\in V\mid x_i>0\}$ and the induced subhypergraph $\cH_W$. By the definition of the height, there exists a hyperedge $H\in\cH_W$ and an $H$-move $x\to x'$ such that $h_\cH(x')=h_\cH(x)-1$.
By property (A1) there exists a hyperedge $H'\in\cH_W$ that intersects all other edges of $\cH_W$.
By property (D1) we have again a chain $H_0,H_1,\dots ,H_p$ such that $H_0=H$, $H_p=H'$ and $|H_{k+1}\setminus H_k|\leq 1$ for all $k=0,1,\dots ,p-1$. Similarly to the above construction, we have a series of $H_k$-moves $k=0,\dots,p$ such that the range of $h_\cH$ values includes all integers $0\leq z < h_\cH(x)$.
\qed

\begin{lemma}\label{l12}
If a hypergraph $\cH\subseteq 2^V$ satisfies properties {\rm (A1)} and {\rm (D1)}, then it also satisfies property {\rm (C2)},
that is, for every position $x\in\ZZP^V$ and integer $1\leq \eta < y_\cH(x)$ there exists a move $x\to x'$ such that $m(x')=m(x)$ and $y_\cH(x')=\eta$.
\end{lemma}

\proof
Let us fix a position $x\in\ZZP^V$ and define $W=\{i\in V\mid x_i>m(x)\}$.
Since $y_\cH(x)+1 = h_\cH(x-m(x)e)$ is the height, there exists a hyperedge $H\in\cH_W$ such that $y_\cH(x^{s(H)})=y_\cH(x)-1$.
By property (A1) there exists a hyperedge $H'\in\cH_W$ that intersects all hyperedges of $\cH_W$. Then by property (D1) there exists a chain $H_0,\dots ,H_p$ in $\cH_W$ such that $H_0=H$ and $H_p=H'$.

Then let $x^{\ga,0}=x^{s(H_0)}$, and define $x^{\go,0}$ by
\[
x^{\go,0}_i=\begin{cases} m(x) & \text{ for } i\in H_0\\ x_i &\text{ otherwise,}\end{cases}
\]
and positions $x^{\ga,k}$ and $x^{\go,k}$ for $k=1,\dots ,p$ by
\[
x^{\ga,k}_i~=~\begin{cases} m(x)&\text{ if } i\in H_{k-1}\cap H_k\\x_i-1&\text{ if } i\in H_k\setminus H_{k-1}\\x_i &\text{ if } i\not\in H_k,\end{cases}
~~~\text{ and }~~~
x^{\go,k}_i~=~\begin{cases} m(x)&\text{ if } i\in H_k\\
x_i &\text{ if } i\not\in H_k.\end{cases}
\]

Let us observe next that the sets of $y_\cH(x')$ values for the ranges $x^{\ga,k}\geq x' \geq x^{\go,k}$, $k=0,1,\dots ,p$ form intervals by Lemma \ref{l2}. Namely, we have
\[
\begin{array}{c}
\left\{y_{\cH}(x')\mid x^{\ga,k}\geq x' \geq x^{\go,k} \right\} = \left[ y_\cH(x^{\go,k}), y_\cH(x^{\ga,k}) \right] \text{ for } k=0,1,\dots ,p.
\end{array}
\]
We claim that these intervals cover all values in the interval $[1,y_\cH(x)-1]$.
To see this claim, we show the following relations:
\[
\begin{array}{lll}
y_\cH(x^{\go,p}) &=&1, \\
y_\cH(x^{\ga,k}) &\geq& y_\cH(x^{\go,k-1})-1 ~~~\text{ for } 1\leq k \leq p, \textrm{ and }\\
y_\cH(x^{\ga,0}) &=&y_\cH(x)-1.
\end{array}
\]
The second group of inequalities follow by Corollary \ref{c0},
since these $y_\cH$-values are essentially the height values by definition \eqref{e-y}.
Indeed, for $i\not\in H_k\setminus H_{k-1}$ we have $x^{\ga,k}_i\geq x^{\go,k-1}_i$ and for the unique index $i\in H_k\setminus H_{k-1}$ we have $x^{\ga,k}_i=x_i-1$ and $x^{\go,k-1}_i=x_i$, and thus Corollary \ref{c0} is applicable.

The first equality is true, since $H_p=H'$ intersects all hyperedges of $\cH_W$, and thus, we have $h_\cH(x^{\go,p}-m(x)e)=0$.
The last equality $y_\cH(x^{\ga,0})=y_\cH(x^{s(H)})=y_\cH(x)-1$ follows by our choice of the set $H$.
\qed

\begin{lemma}\label{l13}
If a hypergraph $\cH\subseteq 2^V$ satisfies properties {\rm (A1)} and {\rm (D1)}, then it also satisfies property {\rm (C3)},
that is, for every position $x\in\ZZP^V$ and integers $0\leq \mu < m(x)$ and
$m(x)-\mu\leq \eta \leq y_\cH(x)$ there exists a move $x\to x'$ such that $m(x')=\mu$ and $y_\cH(x')=\eta$.
\end{lemma}

\proof
Let us fix a position $x\in\ZZP^V$ and assume that $x_j=m(x)$. Let us further fix an integer $0\leq \mu < m(x)$.

Let us first choose an arbitrary hyperedge $H\in\cH$ such that $j\in H$. By property (A1) there exists another hyperedge $H'\in\cH$ that intersects all hyperedges of $\cH$ that do not contain element $j\in V$. Then by property (D1) there exists a chain $H_0,\dots ,H_p$ such that $H_0=H$ and $H_p=H'$.

Let us then define $x^{\ga,0}$ and $x^{\go,0}$ by
\[
x^{\ga,0}_i = \begin{cases} \mu &\text{ if } i=j\\
x_i-1 &\text{ if } i\in H_0\setminus\{j\}\\
x_i &\text{ otherwise},\end{cases} ~~~\text{ and }~~~
x^{\go,0}_i=\begin{cases} \mu & \text{ for } i\in H_0\\ x_i &\text{ otherwise,}\end{cases}
\]
and define positions $x^{\ga,k}$ and $x^{\go,k}$ for $k=1,\dots ,p$ by
\[
x^{\ga,k}_i=\begin{cases} \mu&\text{ if } i\in H_{k-1}\cap H_k\\
x_i-1&\text{ if } i\in H_k\setminus H_{k-1}\\
x_i &\text{ if } i\not\in H_k,\end{cases}
~~~\text{ and }~~~
x^{\go,k}_i=\begin{cases} \mu&\text{ if } i\in H_k\\
x_i &\text{ if } i\not\in H_k.\end{cases}
\]

Let us observe next that the sets of $y_\cH(x')$ values for the ranges $x^{\ga,k}\geq x' \geq x^{\go,k}$, $k=0,1,\dots ,p$ form intervals by Lemma \ref{l2}. Namely, we have
\[
\begin{array}{c}
\left\{y_\cH(x')\mid x^{\ga,k}\geq x'\geq x^{\go,k} \right\} = \left[ y_\cH(x^{\go,k}), y_\cH(x^{\ga,k}) \right] \text{ for } k=0,1, \dots ,p.
\end{array}
\]
We claim that these intervals cover all values in the interval $[1,y_\cH(x)-1]$, as stated in the lemma.
To see this claim, we prove the following relations.
\[
\begin{array}{ll}
y_\cH(x^{\go,p}) &=1,\\
y_\cH(x^{\ga,k}) &\geq~ y_\cH(x^{\go,k-1})-1 ~~~\text{ for } 1\leq k \leq p, \text{ and }\\
y_\cH(x^{\ga,0}) &\geq y_\cH(x).
\end{array}
\]
The second group of inequalities follow by Corollary \ref{c0}, since these $y_\cH$-values are essentially the height by the definition \eqref{e-y}.
This is because for $i\not\in H_k\setminus H_{k-1}$ we have $x^{\ga,k}_i\geq x^{\go,k-1}_i$ and for the unique index $i\in H_k\setminus H_{k-1}$ we have $x^{\ga,k}_i=x_i-1$ and $x^{\go,k-1}_i=x_i$.

The first equality holds since $H_p=H'$ intersects all hyperedges that avoids element $j\in V$, and thus we have $h_\cH(x^{\go,p}-\mu e)=0$.
The last inequality holds since $x^{\ga,0}-\mu e\geq x-m(x) e$.
\qed

\medskip

\noindent{\textbf{Proof of Theorem \ref{t2}}:}
Clearly, property (A1) is stronger than property (A0). Lemmas \ref{l11}, \ref{l12} and \ref{l13} imply that properties (B1), (C2), and (C3) hold.
Thus, the statement follows by Corollary \ref{c1}.
\qed

\section{Classes of JM Hypergraphs}
\label{s4}
In this section we apply Theorem \ref{t2} to a variety of hypergraph classes and show that they are JM.
A small difficulty in this is due to the fact that condition (D2) depends not only on the hypergraph but also on the position. To make our proofs conceptually simpler, we introduce another condition that depends only on the structure of the given hypergraph, and prove that this condition implies (D2):

\begin{itemize}
\item[(D3)] For every subfamily $\cF\subseteq \cH\subseteq 2^V$ such that $V(\cF)\neq V$ there exist hyperedges $F\in\cF$ and $S\in\cH$ such that $\emptyset\neq (S\setminus F)\subseteq V\setminus V(\cF)$.
\end{itemize}

\begin{lemma}\label{l10}
If a hypergraph $\cH\subseteq 2^V$ satisfies property {\rm (D3)}, then it satisfies {\rm (D2)} as well.
\end{lemma}

\proof
Let us fix a position $x\in\ZZP$ for which $m(x)>0$ holds.
An equivalent way of saying property (D2) is that there exists a hyperedge $H\in\cH$ such that $x\to x^{s(H)}$ is a height move and
that for some $i\in H$ we have $x_i=m(x)$.

To prove the lemma let us assume indirectly that this is not the case. In other words, let us introduce
\[
\cF ~=~ \{H\in \cH \mid h_\cH(x^{s(H)})=h_\cH(x)-1\},
\]
and assume indirectly that $x_i>m(x)$ for all $i\in V(\cF)$, implying in particular that $V(\cF)\neq V$.

By property (D3) we have hyperedges $F\in\cF$ and $S\in\cH$ such that $S\setminus F\neq \emptyset$ and
$S\setminus F \subseteq V\setminus V(\cF)$. Let us now consider a mapping $\mu:\cH\to\ZZP$ that defines a height sequence for position $x$,
that is, we have $\sum_{H\in\cH} \mu(H)=h_\cH(x)$ and $x'=\sum_{H\in\cH} \mu(H)\chi(H) \leq x$. In other words, $\mu(H)$ is the number of $H$-moves in this height sequence.
Since $F\in\cF$, we can choose $\mu$ such that $\mu(F)>0$. Note that by our assumption, we have for any $i\in V\setminus V(\cF)$ that
$x_i'=0<m(x)\leq x_i$. Thus, if we define

\[
\mu'(H)~=~ \begin{cases} \mu(F)-1 &\text{ if } H=F\\1 &\text{ if } H=S\\\mu(H) &\text{ otherwise,}\end{cases}
\]
then $\mu'$ also defines a height sequence, with $\mu'(S)>0$ in contradiction with the fact that $S\not\subseteq V(\cF)$.
This contradiction proves that our indirect assumption is not true, that is, there exists an $i\in V(\cF)$ with $x_i=m(x)$, from which the claim of the lemma follows.
\qed

In the sequel we need to argue only about the structure of our constructions, and show that properties (A1), (D1) and (D3) hold.

\subsection{Matroid hypergraphs}

In this section, we study JM matroid hypergraphs.

Let us call $\cH\subseteq 2^V$ a \emph{matroid hypergraph} if
the following exchange property holds for all pairs of hyperedges $H,H'\in\cH$:
\begin{equation}\label{e-exchangeaxiom}
\forall i\in H\setminus H' ~ \ \exists j\in H'\setminus H:
\ \ (H\setminus\{i\})\cup\{j\} \in\cH.\tag{M}
\end{equation}
In other words, $\cH$ is a matroid hypergraph
if it is the set of bases of a matroid (see \cite{Wel76,Wel2010}).

\begin{lemma}\label{lemma-matroid-1ac}
Matroid hypergraphs satisfy {\rm (D1)} and {\rm (D3)}.
\end{lemma}

\proof
For property (D1),  let us fix two arbitrary hyperedges $H,H'\in\cH$,
and let us consider a chain $\cC=(H_0,\dots ,H_p)$ of hyperedges contained in $H\cup H'$ such that
$H_0=H$ and $d(\cC)=|H'\setminus H_p|$ is as small as possible.
Since there are only finitely many different chains in $\cH$, the quantity $d(\cC)$ is well defined.
We claim that $d(\cC)=0$, which implies property (D1), since this applies to any two hyperedges.
To see this claim, assume that $d(\cC)>0$, and apply the exchange axiom for sets $H_p$ and $H'$.
Since $d(\cC)>0$ we have an element $i\in H_p\setminus H'$, and by axiom (M) there exists an element
$j\in H'\setminus H_p$ such that $H_{p+1}=(H_p\setminus\{i\})\cup\{j\}\in\cH$.
Then for $\cC'=(H_0,\dots ,H_p,H_{p+1})$ it follows that $d(\cC')=d(\cC)-1$, contradicting the fact that $d(\cC)$ is as small as possible.
This contradiction proves our claim.

For property (D3), let us consider an arbitrary subfamily $\cF\subseteq \cH$ such that $V(\cF)\neq V$,
and choose two distinct sets $S\in\cH$ and $F\in\cF$ such that $S\not\subseteq V(\cF)$ and
$|(S \setminus F)\cap V(\cF)|$ is as small as possible.
We claim that $|(S \setminus F)\cap V(\cF)|=0$, and hence these sets $S$ and $F$ show property (D3).

To see this claim, assume that there exists an element $i\in (S \setminus F)\cap V(\cF)$.
By the exchange axiom there exists a $j\in
F\setminus S$ such that $S'=(S\setminus\{i\})\cup \{j\}\in \cH$. Then we have
$S'\setminus V(\cF)=S\setminus V(\cF)\neq \emptyset$ and $(S'\setminus V(\cF))\cap V(\cF)= ((S\setminus V(\cF))\cap V(\cF))\setminus\{i\}$.
Since this contradicts the minimality of $|(S \setminus F)\cap V(\cF)|$ our claim follows.
\qed

\begin{theorem}\label{theorem-matroid-1ac}
Let $\cH$ be a matroid hypergraph.
Then $\cH$ is a JM hypergraph if and only if it is minimal transversal-free.
\end{theorem}

\proof
It follows from Theorem \ref{t2} and Lemmas \ref{l3} and \ref{lemma-matroid-1ac}. \qed

\begin{corollary}\label{t1}
Self-dual matroid hypergraphs are JM.
\end{corollary}

\proof
We apply  Theorem \ref{theorem-matroid-1ac}, and show that self-dual matroid hypergraphs are minimal transversal-free.

Let $\cH$ be a self-dual matroid.
Since for every hyperedge in $\cH$, its complement is also a hyperedge,
no $H \in \cH$ is a transversal of $\cH$.
Furthermore, by self-duality of $\cH$, any hyperedge $H \in \cH$ has size $k=n/2$.
In any proper induced subhypergraph on at most $2k-1$ elements any two hyperedges of size $k$ must intersect.
Therefore $\cH$ is minimal transversal-free. \qed

Recall that any matroid hypergraph is $k$-uniform for some $k$.
If $k > n/2$, then it is not transversal-free, and hence not JM.
If $k=n/2$, then we can see that a matroid hypergraph is JM if and only if it is self-dual.
For $k < n/2$, we remark that no matroid hypergraph is self-dual.
However, the following discussion shows that there are a number of JM matroid hypergraphs.

Let $V=\{1,\dots , 7\}$, and define $\cH\subseteq 2^V$ by $\cH={{V}\choose{3}} \setminus \cF$, where
$\cF$ denotes Fano plane, i.e.,
\[\cF=\{
\{1,2,3\},
\{1,4,5\},
\{2,4,6\},
\{1,6,7\},
\{2,5,7\},
\{3,4,7\},
\{3,5,6\}
\}.
\]
Then we can see that $\cH$ is a matroid hypergraph and minimal transversal-free.

We extend the above example and show that there exists a large family of matroid hypergraphs
with $n=2k +\delta$ for  $\delta >0$, which are minimal transversal-free and hence JM.

Let $\delta$ and $k$ be integers such that $0< \delta \leq k-3$, and
assume that $V=\{1, \dots ,n\}$ where $n=2k+\delta$.
Let us consider a $(k+\delta -1)$-uniform hypergraph $\cK \subseteq 2^V$ and associate to it the hypergraph

\begin{equation}\label{e-K}
\cH=\cH(\cK)=\binom{V}{k} \setminus \left\{\binom{K}{k} \mid K \in \cK\right\}.
\end{equation}

We shall also consider the following conditions:
\begin{itemize}
\item[(K1)]$|K \cap K'| \geq \delta$ for all hyperedges $K,K' \in \cK$.
\item[(K2)]$|K \cap K'| \leq k-2$ for all distinct hyperedges $K,K' \in \cK$.
\item[(K3)] No singleton is a transversal of $\cK$.
 \end{itemize}
Define

\begin{lemma}
\label{lemma-k2ae}
Let $\cK$ be a $(k+\delta -1)$-uniform hypergraph that satisfies {\rm (K2)},
 $S \subseteq V$ be a set of size $|S|=k-1$, and  $W \subseteq V \setminus S$.
Assume that for any $v \in W$, $S \cup \{v\}$ is not contained in $\cH$.
Then  we have $|W|\leq \delta$.
\end{lemma}

\proof
Since $S \cup \{v\}$ is not contained in $\cH$ for every $v \in W$, there exists $K_v$  in $\cK$ such that $K_v \supseteq S \cup \{v\}$.
Note that the union of the sets $K_v$ is of size at least $k-1+|W|$,
and thus if $|W|\geq \delta +1$,  we have two elements $v$ and $v'$ in $W$ such that $K_v\not= K_{v'}$.
However, since $K_v\cap K_{v'}\supseteq S$, this contradicts (K2). \qed

\begin{lemma}
\label{lemma-exmat1}
If a $(k+\delta -1)$-uniform hypergraph $\cK$ satisfies {\rm (K1)}, {\rm (K2)}, and {\rm (K3)}, then  $\cH$ is minimal transversal-free.
\end{lemma}

\proof
We first show that $\cH$ is transversal-free, i.e.,
any hyperedge $H \in \cH$ has a hyperedge $H' \in \cH$ such that $H \cap H'=\emptyset$.
Let $S \subseteq V \setminus H$ with $|S|=k-1$, and $W= V \setminus (H \cup S)$.
Since $|W| > \delta$,  by  Lemma \ref{lemma-k2ae} we must have at least one $v \in W$  such that
$H'=S \cup \{v\}$ belongs to $\cH$.

We next show that for any nonempty set $R \subseteq V$, the induced subhypergraph $\cH_{V \setminus R}$ is either empty or it has a transversal
$T \in  \cH_{V \setminus R}$.

If $|R| \geq \delta+1$, then any two hyperedges in $\cH_{V \setminus R}$
intersect.
Thus it remains to consider the case of $|R| \leq \delta$.
Let $K \in \cK$ be a hyperedge that maximizes $|R \setminus K|$.
Then by (K3) we have $R \not\subseteq K$, and thus  $|V\setminus (R \cup K)|\leq k$ is implied.

If $|V\setminus (R \cup K)|< k$, let $S$ be a set such that $V\setminus (R \cup K) \subseteq S\subseteq V \setminus R$ and $|S|=k-1$, and let $W=K\setminus (R \cup S)$.
Since $|W| = k+\delta+1-|R| > \delta$,
by  Lemma \ref{lemma-k2ae} we must have at least one $v \in W$  such that
$H'=S \cup \{v\}$ belongs to $\cH$.
We claim that
$H'$ is a transversal of $\cH_{V \setminus R}$.
Indeed, $V\setminus (H' \cup R)$ is a subset of $K$, and thus no subset of it (of size $k$) is contained in $\cH$.

On the other hand, if  $|V\setminus (R \cup K)|= k$,
then $H'=V\setminus (R \cup K)$ is a transversal hyperedge in $\cH_{V \setminus R}$.
Indeed, $H' \in \cH_{V \setminus R}$,
since  otherwise there exists a $K' \in \cK$ such that $H' \subseteq K'$.
This  implies
$|K \cap K'| < \delta$, contradicting (K1).
Furthermore, $\cH_{V \setminus R}$ contains no hyperedge disjoint from $H'$,
since $V\setminus (R \cup H')$ is a subset of $K \in \cK$. \qed

\begin{lemma}
\label{lemma-exmat2}
If a $(k+\delta -1)$-uniform hypergraph $\cK$ satisfies {\rm (K2)}, then  $\cH$ is a matroid hypergraph.
\end{lemma}

\proof
Consider two distinct hyperedges $H, H' \in \cH$.
We assume that \raf{e-exchangeaxiom}  does not holds for $H$ and $H'$, and
derive a contradiction.

If $|H \cap H'|=k-1$ then
\raf{e-exchangeaxiom} clearly holds, and hence, we have $|H \cap H'|\leq k-2$.
By our assumption, there exists an element $i \in H\setminus H'$ such that
any $j \in H' \setminus H$ satisfies $(H \setminus \{i\}) \cup \{j\} \not\in \cH$.
This means that  for any $j \in H' \setminus H$, we have $K_j \in \cK$ such that
$K_j \supseteq (H \setminus \{i\}) \cup \{j\}$.
We note that $H' \setminus H$ contains two elements $j,\ell$  that satisfy  $K_j\not=K_\ell$, since otherwise $K_j$ contains $H'$, a contradiction on its own.
Since for these indices $j$ and $\ell$, we have $ |K_j \cap K_\ell|\geq k-1$, we get a contradiction with (K2).
\qed

\begin{theorem}
If a $(k+\delta -1)$-uniform hypergraph $\cK$ satisfies {\rm (K1)}, {\rm (K2)}, and {\rm (K3)},  then  $\cH$, defined by \eqref{e-K}, is a JM and matroid hypergraph.
\end{theorem}

\proof
It follows from Lemmas \ref{lemma-exmat1} and \ref{lemma-exmat2}. \qed

To complete this section, let us construct a hypergraph $\cK$ with the desired properties.

Let $\delta$ and $k$ be integers such that $0< \delta \leq k-3$.
Define $V$ by $V=W \cup U$, where $W=\{1, \dots, k+\delta-1\}$, and $U=\{1', \dots, (k+1)'\}$.
Note that $|V|=(k+\delta -1) + (k+1)=2k +\delta$.
Let
\[
\cK=\{W, \{1,\dots , \delta\} \cup \{1', \dots , (k-1)'  \},
\{\delta+1,\dots , 2\delta\} \cup \{3', \dots , (k+1)'  \}\}.
\]
It is not difficult to see that $\cK  \subseteq 2^V$
is a $(k+\delta -1)$-uniform hypergraph satisfying (K1), (K2), and (K3).

\subsection{JM Hypergraphs Arising from Graphs}\label{ss-BogPal}

In this section we use the standard terminology of graph theory,
 see e.g., \cite{Har69}. We shall consider simple undirected graphs, and
use their set of edges or set of vertices as the base set to define additional families of JM hypergraphs.

Given a graph $G=(U,E)$ and a subset $F\subseteq E$ of the edges, we denote by $d_F(u)$
the degree of vertex $u\in U$ with respect to the subgraph $(U,F)$.
In other words, $d_F(u)$ is the number of edges in $F$ that are incident with vertex $u$.
If $d_F(u)>0$ then we call $u$ a supporting vertex of subset $F$, and
we denote by $U(F)\subseteq U$ the set of supporting vertices of $F$, that is,
$U(F) ~=~ \{u\in U\mid d_F(u)>0\}$.

Given an integer $k<|E|$, let us define
\begin{eqnarray}
\cF_{e,c}(G,k) &=& \{ F\subseteq E\mid |F|=k, (U(F),F) \text{ is connected} \}\label{ec}\\
\cF_{v,c}(G,k) &=& \left\{ U(F)\subseteq U \left| \begin{array}{c}
F\subseteq E, |F|=k,\\ (U(F),F) \text{ is connected}
\end{array} \right. \right\}\label{vc}\\
\cF_{e,t}(G,k) &=& \{ F\subseteq E\mid |F|=k , (U(F),F) \text{ is a tree} \}\label{et}\\
\cF_{v,t}(G,k) &=& \left\{ U(F)\subseteq U \left| \begin{array}{c}
F\subseteq E, |F|=k,\\ (U(F),F) \text{ is a tree}
\end{array} \right.\right\}. \label{vt}
\end{eqnarray}

\begin{lemma}\label{l-connected-chain-D2}
If $G=(U,E)$ is a connected graph and $1\leq k<|E|$, then the hypergraphs $\cF_{e,c}(G,k)$, $\cF_{e,t}(G,k)$, $\cF_{v,c}(G,k)$, and $\cF_{v,t}(G,k)$ satisfy property {\rm (D3)}.
\end{lemma}

\proof
We are going to prove the statement for $\cF_{e,c}(G,k)$. For the others similar proofs work.

Assume that $\cF\subseteq \cF_{e,c}(G,k)$ is a subfamily such that $\bigcup_{F\in\cF}F\neq E$. Let us denote by $W$ the set of vertices incident to some of the sets in $\cF$, and let $e$ be an edge of $G$
incident with $W$
that does not belong to any of the sets in $\cF$. Such an edge $e$ exists by our assumption. Let us denote by $w\in W$ the vertex with which $e$ is incident (note that $e$ maybe incident with two vertices of $W$). Since $w\in W$, there is a set $F\in \cF$ such that $w\in U(F)$. We claim that there exists $f\in F$ such that $S=(F\setminus\{f\})\cup\{e\}\in \cF_{e,c}(G,k)$.
Namely, if $F\cup\{e\}$ contains a cycle then any $f\neq e$ of this cycle can be chosen, otherwise $F\cup \{e\}$ is a tree, and therefore it must have a leaf edge $f\neq e$.
Then the pair of sets $F$ and $S$ proves that property (D3) holds.
\qed

\begin{lemma}\label{l-connected-chain-D1}
If $G=(U,E)$ is a connected graph and $k<|E|$,
then the hypergraphs $\cF_{e,c}(G,k)$, $\cF_{e,t}(G,k)$ for $k\geq 2$ and
$\cF_{v,c}(G,k)$, $\cF_{v,t}(G,k)$ for $k\geq 1$ satisfy property {\rm (D1)} if they are minimal transversal free.
\end{lemma}

\proof
We are going to prove the statement for $\cF_{e,c}(G,k)$. For the others similar proofs work.
Let us consider arbitrary hyperedges $A,B \in \cF_{e,c}(G,k)$. Let us notice first that since $G$ is connected and $\cF_{e,c}(G,k)$ is minimal transversal free, $A\cup B$ induces a connected subgraph $G[A\cup B]$ of $G$. Consequently, $U(A)\cap U(B)\neq\emptyset$.
Let us denote by
$d(A,B)$ the size of a maximum connected component in $A\cap B$.
We claim that if $A\neq B$ then there exists a $D \in \cF_{e,c}(G,k)$ such that
$A \cap D \neq \emptyset$, $|D\setminus A|=1$, $D\subseteq A\cup B$, and $d(A,B) < d(D,B)$. By repeatedly applying this claim, we can construct a chain from $A$ to $B$.

Choose a maximum connected component
$K \subseteq A\cap B$. Choose $e \in B\setminus A$ incident with $K$ and $A$.
Such an edge exists since $B$ is connected, $B\neq A$, and $U(A)\cap U(B)\neq\emptyset$. Then there exists $f \in A\setminus K$ such that $D=(A\cup\{e\}\setminus\{f\}$ is connected.
To see this contract edges $K \cup \{e\}$. If $A\setminus K$ contains a cycle after this contraction, then any edge of this cycle can be chosen.
Otherwise we choose a leaf edge $f \in A\setminus K$.

Thus, the set $D$ satisfies the claim.
\qed

\begin{theorem}\label{c-t-ad3}
Let $G=(U,E)$ be a connected graph. If $\cF_{e,*}(G,k)$ for $k\geq 2$ is minimal transversal-free, where $*\in \{c,t\}$, then it is JM. Similarly, if $\cF_{v,*}(G,k)$ for $k\geq 1$ is minimal transversal-free, where $*\in \{c,t\}$, then it is JM.
\end{theorem}

\proof
Property (A1) follows by our assumption, while properties (D1) and (D3) follow by Lemmas \ref{l-connected-chain-D1} and \ref{l-connected-chain-D2}. Thus, the statement is implied by Theorem \ref{t2}.
\qed

\medskip

There are several infinite families of graphs for which we can apply Theorem \ref{c-t-ad3}.
We use standard graph theoretical notation, see e.g., \cite{Har69}.  We denote by $C_n$ the simple cycle on $n$ vertices, $K_{a,b}$ the complete bipartite graph with $a$ and $b$ vertices in the two classes, etc.

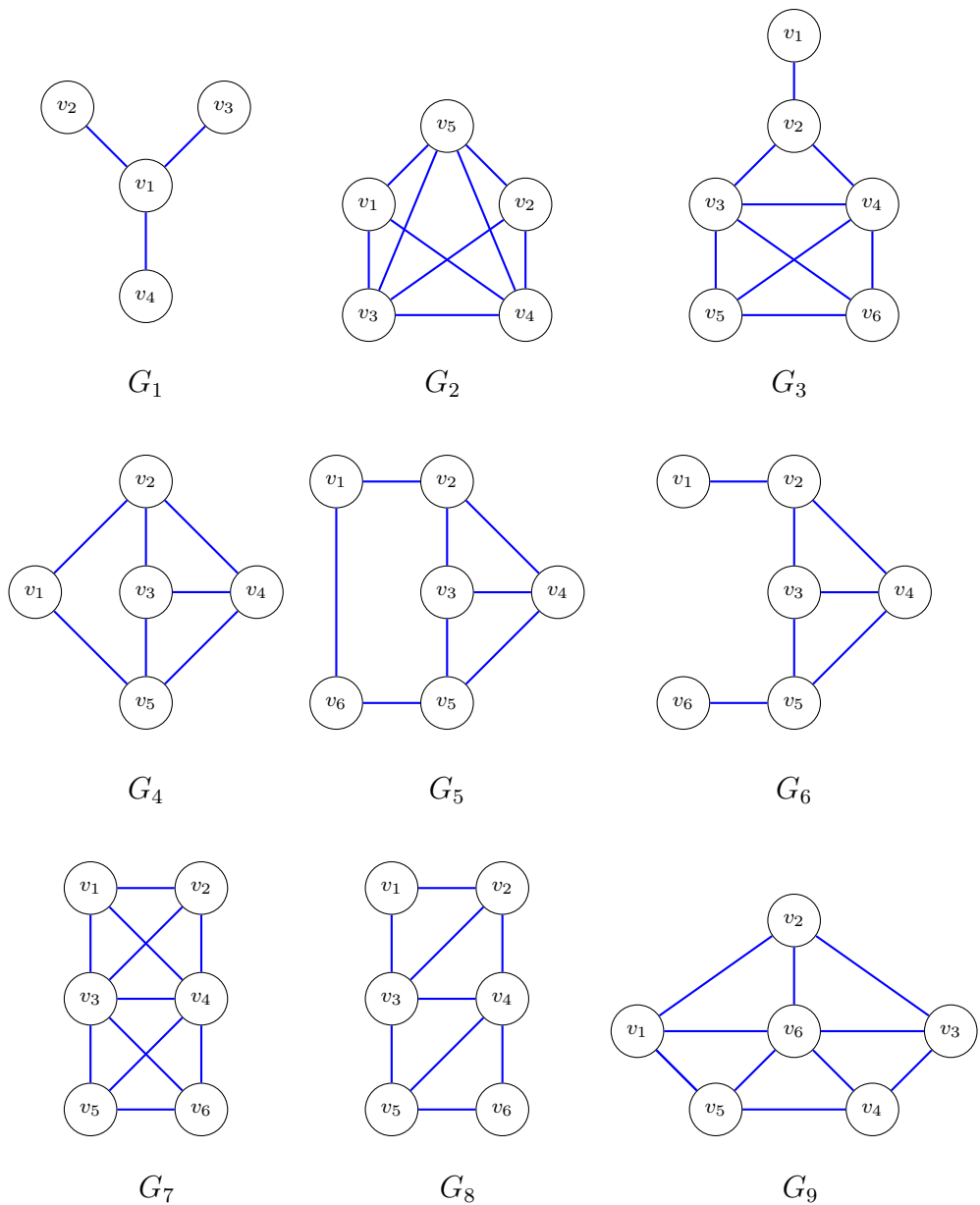
\begin{figure}[htbp]
\centering
\begin{tabular}{ccc}
\begin{tikzpicture}[node distance=1.5cm,scale=1]
\node[draw,circle] (1) {\tiny $v_1$};
\node[draw,circle,above left of = 1] (2){\tiny $v_2$};
\node[draw,circle,above right of = 1] (3) {\tiny $v_3$};
\node[draw,circle,below of = 1] (4) {\tiny $v_4$};

\node[below = 0.5cm of 4] {$G_1$};

\foreach \x/\y in {1/2,1/3,1/4} \draw[blue,thick] (\x) -- (\y);
\end{tikzpicture}
&
\begin{tikzpicture}[node distance=1.5cm,scale=1]
\node[draw,circle] (1) {\tiny $v_1$};
\node[draw,circle,below of = 1] (3){\tiny $v_3$};
\node[draw,circle,above right of = 1] (5) {\tiny $v_5$};
\node[draw,circle,below right of = 5] (2) {\tiny $v_2$};
\node[draw,circle,below of = 2] (4) {\tiny $v_4$};

\node[below right =0.5cm of 3] {$G_2$};

\foreach \x/\y in {1/3,1/4,1/5,2/3,2/4,2/5,3/4,3/5,4/5} \draw[blue,thick] (\x) -- (\y);
\end{tikzpicture}
&
\begin{tikzpicture}[node distance=1.5cm,scale=1]
\node[draw,circle] (1) {\tiny $v_1$};
\node[draw,circle,below =0.5cm of 1] (2){\tiny $v_2$};
\node[draw,circle,below left of = 2] (3) {\tiny $v_3$};
\node[draw,circle,below right of = 2] (4) {\tiny $v_4$};
\node[draw,circle,below of = 3] (5) {\tiny $v_5$};
\node[draw,circle,below of = 4] (6) {\tiny $v_6$};

\node[below right =0.5cm of 5] {$G_3$};

\foreach \x/\y in {1/2,2/3,2/4,3/4,3/5,3/6,4/5,4/6,5/6} \draw[blue,thick] (\x) -- (\y);
\end{tikzpicture}
\\*[5mm]
\begin{tikzpicture}[node distance=1.5cm,scale=1]
\node[draw,circle] (1) {\tiny $v_1$};
\node[draw,circle,right of = 1] (3){\tiny $v_3$};
\node[draw,circle,above of = 3] (2) {\tiny $v_2$};
\node[draw,circle,below of = 3] (5) {\tiny $v_5$};
\node[draw,circle,right of = 3] (4) {\tiny $v_4$};

\node[below =0.5cm of 5] {$G_4$};

\foreach \x/\y in {1/2,1/5,2/3,2/4,3/4,3/5,4/5} \draw[blue,thick] (\x) -- (\y);
\end{tikzpicture}
&
\begin{tikzpicture}[node distance=1.5cm,scale=1]
\node[draw,circle] (1) {\tiny $v_1$};
\node[draw,circle,right of = 1] (2) {\tiny $v_2$};
\node[draw,circle,below of = 2] (3) {\tiny $v_3$};
\node[draw,circle,below of = 3] (5) {\tiny $v_5$};
\node[draw,circle,left of = 5] (6) {\tiny $v_6$};
\node[draw,circle,right of = 3] (4){\tiny $v_4$};

\node[below =0.5cm of 5] {$G_5$};

\foreach \x/\y in {1/2,1/6,2/3,2/4,3/4,3/5,4/5,5/6} \draw[blue,thick] (\x) -- (\y);
\end{tikzpicture}
&
\begin{tikzpicture}[node distance=1.5cm,scale=1]
\node[draw,circle] (1) {\tiny $v_1$};
\node[draw,circle,right of = 1] (2) {\tiny $v_2$};
\node[draw,circle,below of = 2] (3) {\tiny $v_3$};
\node[draw,circle,below of = 3] (5) {\tiny $v_5$};
\node[draw,circle,left of = 5] (6) {\tiny $v_6$};
\node[draw,circle,right of = 3] (4) {\tiny $v_4$};

\node[below =0.5cm of 5] {$G_6$};

\foreach \x/\y in {1/2,2/3,2/4,3/4,3/5,4/5,5/6} \draw[blue,thick] (\x) -- (\y);
\end{tikzpicture}
\\*[5mm]
\begin{tikzpicture}[node distance=1.5cm,scale=1]
\node[draw,circle] (1) {\tiny $v_1$};
\node[draw,circle,right of = 1] (2) {\tiny $v_2$};
\node[draw,circle,below of = 1] (3) {\tiny $v_3$};
\node[draw,circle,right of = 3] (4) {\tiny $v_4$};
\node[draw,circle,below of = 3] (5) {\tiny $v_5$};
\node[draw,circle,right of = 5] (6) {\tiny $v_6$};

\node[below right = 0.5cm and 0.25cm of 5] {$G_7$};

\foreach \x/\y in {1/2,1/3,1/4,2/3,2/4,3/4,3/5,3/6,4/5,4/6,5/6} \draw[blue,thick] (\x) -- (\y);
\end{tikzpicture}
&
\begin{tikzpicture}[node distance=1.5cm,scale=1]
\node[draw,circle] (1) {\tiny $v_1$};
\node[draw,circle,right of = 1] (2) {\tiny $v_2$};
\node[draw,circle,below of = 1] (3) {\tiny $v_3$};
\node[draw,circle,right of = 3] (4) {\tiny $v_4$};
\node[draw,circle,below of = 3] (5) {\tiny $v_5$};
\node[draw,circle,right of = 5] (6) {\tiny $v_6$};

\node[below right = 0.5cm and 0.25cm of 5] {$G_8$};

\foreach \x/\y in {1/2,1/3,2/3,2/4,3/4,3/5,4/5,4/6,5/6} \draw[blue,thick] (\x) -- (\y);
\end{tikzpicture}
&
\begin{tikzpicture}[node distance=1.5cm,scale=1]
\node[draw,circle] (1) {\tiny $v_1$};
\node[draw,circle,below right of = 1] (5) {\tiny $v_5$};
\node[draw,circle,above right of = 5] (6) {\tiny $v_6$};
\node[draw,circle,below right of = 6] (4) {\tiny $v_4$};
\node[draw,circle,above right of = 4] (3) {\tiny $v_3$};
\node[draw,circle,above of = 6] (2) {\tiny $v_2$};

\node[below right = 0.5cm and 0.5cm of 5] {$G_9$};

\foreach \x/\y in {1/2,1/5,1/6,2/3,2/6,3/4,3/6,4/5,4/6,5/1,5/6} \draw[blue,thick] (\x) -- (\y);
\end{tikzpicture}
\end{tabular}
\caption{The nine forbidden induced subgraphs characterizing line graphs, see \cite{Bei70}.\label{fig-1}}
\end{figure}

\begin{description}
\item[Circulant hypergraphs:] For any given value of $k\geq 2$,  it is easy to see that the graphs $C_{2k}$ and $C_{2k+1}$ yield JM hypergraphs with all four definitions.
In fact, they are isomorphic families for both $C_{2k}$ and $C_{2k+1}$.

\item[Additional self-dual matroids:] Graphs $K_{2,k}$ are also good example for $k\geq 2$ with both definitions \eqref{ec} and \eqref{et}. Interestingly, both hypergraphs are self-dual matroids, and they are not isomorphic with one another for $k\geq 4$.

\item[Trees:] For any $k\geq 2$ the following subfamily of trees on $k^2+k$ edges provide good examples with both definitions involving edge subsets.
Let $T_i$, $i=1,\dots ,k+1$ be an arbitrary trees of $k$ edges each on distinct sets of vertices, and let $v_i$ be a leaf vertex of $T_i$ for all $i=1,\dots ,k+1$. Then we can get a tree $T$ by identifying these leaf vertices. $T$ has $k^2+k$ edges, and the family $\cF_{e,c}(T,k)$ is minimal transversal-free. (Note that definitions \eqref{ec} and \eqref{et} yield isomorphic hypergraphs in this case.)

\item[Star of cliques:] Another example is a graph $G$ formed by $k+1$ cliques on $k$ vertices each, joined by one-one edges to a common root vertex. In this case only definition \eqref{et} yields a hypergraph that is, $\cF_{e,t}(G,k)$ minimal transversal-free and has $(k+1)(\binom{k}{2}+1)$ vertices.

\item[Petersen:] Finally, a singular example is provided by the Petersen graph $P$ for which the family $\cF_{e,c}(P,7)$ is minimal transversal-free.
\end{description}

In Section \ref{s-size}, we show that the number of JM hypergraphs defined by  \eqref{ec}, \eqref{vc},  \eqref{et}, and \eqref{vt} is bounded by a function of $k$.

\subsection{JM Graphs}
\label{sec-jmgraph1ac}

In this section we provide a complete characterization of JM {\em graphs}, i.e., $2$-uniform JM hypergraphs.
Note that
\begin{equation}
\label{eq-erty1}
E=\cF_{v,c}(G,1)=\cF_{v,t}(G,1)
\end{equation}
holds for any graph $G=(V,E)$, which implies the following result.

\begin{theorem}
\label{th-jmgraph1w}
A graph $G$ is  JM  if and only if it is connected and minimal transversal-free.
\end{theorem}

\proof
The necessity follows from Lemmas \ref{necessaryconditions} and \ref{l3},
and the sufficiency follows from Theorem \ref{c-t-ad3} and \raf{eq-erty1}.
\qed

In the sequel we characterize all connected minimal transversal-free graphs.

\begin{lemma}\label{jmgarelg}
If a graph $G=(V,E)$ is  connected and minimal transversal-free, then it is the line graph of a simple graph.
\end{lemma}

\proof
Let us indirectly assume that $G$ is not a line graph. Then $G$ must contain one of the $9$ forbidden induced subgraphs shown in \cite{Bei70}, see Figure \ref{fig-1}. We claim that none of these $9$ graphs can be an induced subgraph of $G$, since we assumed that $G$ is minimal transversal-free.
For this, note that in each of the graphs $G_i$, $i=2,\dots ,9$ contains as a proper induced subgraph at least one of $C_4$, $C_5$, $K_4$, or $2K_2$. In all cases we do not have an edge that would intersect all others.

We claim  that  the claw $G_1$ cannot be an  induced subgraph of $G$.
Since $G_1$ is not transversal-free, we assume that $G_1$ is a proper induced subgraph of $G$, and derive a contradiction.
Remove from $G$ a leaf of the claw, say $v_2$.
Then $G\setminus v_2$ has an intersecting edge $e$, i.e.,
an edge $e$ in $G\setminus v_2$ intersects all edges in $G\setminus v_2$.
We note that $e$ is incident with the center of the claw $v_1$.

Suppose that $e$ is an edge of the claw, say $e=(v_1,v_3)$. Then $v_4$ is a vertex of degree $1$ in $G$. Let us denote by $e'$ the intersecting edge in $G\setminus v_4$.
 The edge $e'$ must be incident with the center $v_1$, and hence, $e'$ is an intersecting edge of $G$, which contradicts that $G$ is transversal-free.

Suppose now that $e$ is not contained in the claw, say $e=(v_1,v_5)$ where $v_5$ is not in $G_1$. If $v_3$ or $v_4$ is of degree $1$ then by the same argument $G$ has an intersecting edge. Otherwise $G$ contains edge $e_1=(v_3,v_5)$. Since $G$ is transversal-free, there must exist an edge $e_2$ disjoint from $e$ through $v_2$, say $e_2=(v_2,v_6)$. Consider again $G \setminus v_4$. By our assumption it has an intersecting edge $e_3$. As we know, $e_3$ is incident with $v_1$. However, no such edge can intersect both $e_1$ and $e_2$. This contradicts that $e_3$ is an intersecting edge in $G\setminus v_4$.
\qed


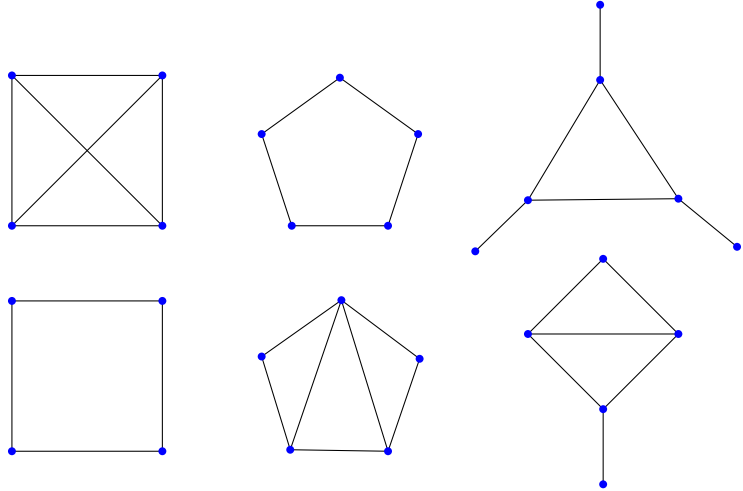
\begin{figure}[htbp]
\centering
\definecolor{qqqqff}{rgb}{0,0,1}
\begin{tikzpicture}[line cap=round,line join=round,>=triangle 45,x=1.0cm,y=1.0cm]
\clip(0.2,0.48) rectangle (11.37,7.06);
\draw (1,6)-- (3,6);
\draw (3,4)-- (1,4);
\draw (1,6)-- (1,4);
\draw (3,6)-- (3,4);
\draw (3,4)-- (1,6);
\draw (3,6)-- (1,4);
\draw (1,3)-- (3,3);
\draw (1,3)-- (1,1);
\draw (1,1)-- (3,1);
\draw (3,3)-- (3,1);
\draw (7.86,2.56)-- (8.86,3.56);
\draw (9.86,2.56)-- (8.86,3.56);
\draw (7.86,2.56)-- (9.86,2.56);
\draw (9.86,2.56)-- (8.86,1.56);
\draw (8.86,1.56)-- (7.86,2.56);
\draw (8.86,1.56)-- (8.86,0.56);
\draw (7.16,3.66)-- (7.86,4.34);
\draw (7.86,4.34)-- (8.82,5.94);
\draw (9.86,4.36)-- (8.82,5.94);
\draw (7.86,4.34)-- (9.86,4.36);
\draw (9.86,4.36)-- (10.64,3.72);
\draw (8.82,5.94)-- (8.82,6.94);
\draw (5.38,3.01)-- (4.7,1.02);
\draw (5.38,3.01)-- (6,1);
\draw (4.32,5.22)-- (4.72,4);
\draw (4.72,4)-- (6,4);
\draw (6,4)-- (6.4,5.22);
\draw (6.4,5.22)-- (5.36,5.97);
\draw (5.36,5.97)-- (4.32,5.22);
\draw (5.38,3.01)-- (4.32,2.26);
\draw (4.32,2.26)-- (4.7,1.02);
\draw (4.7,1.02)-- (6,1);
\draw (6,1)-- (6.42,2.23);
\draw (5.38,3.01)-- (6.42,2.23);
\begin{scriptsize}
\fill [color=qqqqff] (1,1) circle (1.5pt);
\fill [color=qqqqff] (3,1) circle (1.5pt);
\fill [color=qqqqff] (3,3) circle (1.5pt);
\fill [color=qqqqff] (1,3) circle (1.5pt);
\fill [color=qqqqff] (1,4) circle (1.5pt);
\fill [color=qqqqff] (3,4) circle (1.5pt);
\fill [color=qqqqff] (3,6) circle (1.5pt);
\fill [color=qqqqff] (1,6) circle (1.5pt);
\fill [color=qqqqff] (4.7,1.02) circle (1.5pt);
\fill [color=qqqqff] (6,1) circle (1.5pt);
\fill [color=qqqqff] (4.72,4) circle (1.5pt);
\fill [color=qqqqff] (6,4) circle (1.5pt);
\fill [color=qqqqff] (7.86,2.56) circle (1.5pt);
\fill [color=qqqqff] (8.86,3.56) circle (1.5pt);
\fill [color=qqqqff] (9.86,2.56) circle (1.5pt);
\fill [color=qqqqff] (8.86,1.56) circle (1.5pt);
\fill [color=qqqqff] (8.86,0.56) circle (1.5pt);
\fill [color=qqqqff] (8.82,5.94) circle (1.5pt);
\fill [color=qqqqff] (7.86,4.34) circle (1.5pt);
\fill [color=qqqqff] (9.86,4.36) circle (1.5pt);
\fill [color=qqqqff] (7.16,3.66) circle (1.5pt);
\fill [color=qqqqff] (8.82,6.94) circle (1.5pt);
\fill [color=qqqqff] (10.64,3.72) circle (1.5pt);
\fill [color=qqqqff] (6.4,5.22) circle (1.5pt);
\fill [color=qqqqff] (5.36,5.97) circle (1.5pt);
\fill [color=qqqqff] (4.32,5.22) circle (1.5pt);
\fill [color=qqqqff] (6.42,2.23) circle (1.5pt);
\fill [color=qqqqff] (5.38,3.01) circle (1.5pt);
\fill [color=qqqqff] (4.32,2.26) circle (1.5pt);
\end{scriptsize}
\end{tikzpicture}
\caption{The six JM graphs}\label{6JMgraphs}
\end{figure}

The following statement is  straightforward from  the definition.

\begin{lemma}\label{ec2samelg}
For every simple graph $G$ we have that $\cF_{e,c}(G,2)$ is the edge set of the line graph of $G$.\qed
\end{lemma}

\begin{lemma}\label{4to6}
If $G=(V,E)$ is a JM graph then we have $4\leq |V|\leq 6$.
\end{lemma}

\proof
It follows from Theorem \ref{th-jmgraph1w} that $G$ has at least two disjoint edges.
Hence, we have $|V|\geq 4$.
By Theorem \ref{th-jmgraph1w} and Lemmas \ref{jmgarelg} and \ref{ec2samelg}, there exists a graph $G^*=(V^*,E^*)$ such that $E= \cF_{e,c}(G^*,2)$.
By Theorem \ref{th-jmgraph1w}  and Lemma \ref{ecbound} we have $|E^*|\leq 2^2+2=6$,
where Lemma \ref{ecbound} can be found in the next section.
Since $G$ is the line graph of $G^*$, we have $E^*=V$, implying $|V| \leq 6$.
\qed

\begin{theorem}
Among all graphs, only six graphs in Figure \ref{6JMgraphs} are JM.
\end{theorem}

\begin{proof}
It is easy to see that all six graphs are connected and minimal transversal-free. Since graphs correspond to $\cF_{v,c}(G,1)$, by Theorem \ref{c-t-ad3} they are JM graphs.

To show that no other JM graph exists, it is sufficient to check all graphs $G=(V,E)$ with $4 \leq |V| \leq 6$ by Lemma \ref{4to6}; see e.g., \cite{Har69} for a complete list of graphs with up to $6$ vertices.
\qed
\end{proof}

Before concluding this section, we remark that Lemma \ref{4to6} provides a tighter bound than the one in (iii) of Lemma \ref{ecbound}, when $k=1$.

\section{Size of $k$-Uniform JM Hypergraphs}
\label{s-size}

In this section we study the bound for the size of $k$-uniform JM hypergraphs.
We first provide upper bound for the size of $k$-uniform minimal transversal-free hypergraphs, implying upper bound for the size of $k$-uniform JM hypergraphs.

\begin{lemma}[\cite{Bog17}]\label{Bogdanov}
Assume that $\cH\subseteq 2^V$ is a $k$-uniform minimal transversal-free hypergraph. Then, we have
\[
|V| ~\leq~ k\binom{2k}{k}.
\]
\end{lemma}
\proof
Since $\cH$ is a minimal transversal-free, for every $i\in V$ we have a hyperedge $H_i\in\cH$ such that $H_i\cap H'\neq\emptyset$ for all $H'\in\cH$ with $i\not\in H'$. Let us denote by $\cH'=\{H_i\mid i\in V\}$ the family of these hyperedges. By the transversal-freeness we also have for every hyperedge $H\in \cH'$ a disjoint hyperedge $B(H)\in\cH$, $H\cap B(H)=\emptyset$. Let us now choose a minimal subhypergraph $\cB\subseteq \cH$ such that
\begin{equation}
\label{e-PP}
\forall H\in\cH' ~~\exists B\in\cB: ~~H\cap B=\emptyset.
\end{equation}
Let us note first that such a $\cB$ must form a cover of $V$, i.e., $V=\bigcup_{B\in\cB}B$. This is because for all $H_i\in\cH'$ we have a $B\in \cB$ such that $H_i\cap B=\emptyset$, and consequently, $i\in B$. Let us observe next that for all $B\in \cB$ we have at least one $A(B)\in\cH'$ such that $A(B)\cap B=\emptyset$ and $A(B)\cap B'\neq\emptyset$ for all $B'\in\cB\setminus\{B\}$. This is because we choose $\cB$ to be a minimal family with respect to \eqref{e-PP}.
Let us now define $\cA=\{A(B)\in\cH'\mid B\in\cB\}$. The pair $\cA$, $\cB$ of  hypergraphs now satisfies the conditions of a classical theorem of Bollob\'as \cite{Bol65}, which then implies that
\[
|\cA|=|\cB|\leq\binom{2k}{k}.
\]
Since $\cB$ is a $k$-uniform hypergraph that covers $V$, our claim follows.
\qed

This clearly implies that $|\cH| \leq 2^{k\binom{2k}{k}}$.

An example, provided by D. P\'alv\"olgyi \cite{Pal17}, almost matches the upper bound above on the size of $V$, and we recall it here for completeness.

Let $V=U\cup W$, where $|U|=2k-2$, $|W|=\frac12 {{2k-2}\choose{k-1}}$, and $U\cap W=\emptyset$.

Consider all $(k-1)$ subsets of $U$, labeled as
$A_i$ and $B_i$ such that $A_i\cap B_i=\emptyset$ for $i=0,\dots ,r-1$,
where $r=\frac12 {{2k-2}\choose{k-1}}$. Assume further that $W=\{w_0,w_1,\dots ,w_{r-1}\}$, and
define
\[
\cH=\{B_i\cup\{w_i\},A_{i+1}\cup \{w_i\}\mid i=0,\dots ,r-1\},
\]
where indices are taken modulo $r$.
The hypergraph $\cH$ is $k$-uniform.

Easy to see that it is a minimal transversal-free hypergraph. Namely, if we delete some points from $U$
then all remaining hyperedges are intersecting already in $U$.
If we delete some points from $W$ but not $U$ then consider an index $i$
such that we deleted $w_i$ and not $w_{i+1}$.
Then $B_{i+1}\cup\{w_{i+1}\}$ intersects all remaining hyperedges.

\medskip

We next consider the size of JM hypergraphs discussed in Section \ref{s3}.
As mentioned in the introduction,  self-dual matroid hypergraphs $\cH$ are $k$-uniform for $k=n/2$, and satisfy
\[
2^k \leq |\cH| \leq \binom{2k}{k}.
\]
Since we characterize JM graphs in Section \ref{sec-jmgraph1ac},
in the rest of this section, we provide an upper bound for the size of JM hypergraphs definied by  \eqref{ec}, \eqref{vc}, \eqref{et},  and \eqref{vt}.
For this, we  prove that the size of a graph, for which definitions \eqref{ec}, \eqref{vc}, \eqref{et},  and \eqref{vt}   yield minimal transversal-free hypergraphs, is bounded by a function of $k$.

\begin{lemma}\label{ecbound}
Let $G=(U,E)$ be a connected graph.
\begin{itemize}
\item[(i)] If $\cF_{e,c}(G,k)$ is minimal transversal-free, then $|E|\leq k^2+k$.
\item[(ii)] If $G$ is simple and $\cF_{e,t}(G,k)$ is minimal transversal-free, then $|E|\leq \frac{k^3}{2}+\frac{k}{2}+1$.
\item[(iii)] If $\cF_{v,c}(G,k)$ or $\cF_{v,t}(G,k)$ is minimal transversal-free, then
$|U| \leq 2k^3+4k^2+k+2$.
\end{itemize}
\end{lemma}

\proof
We prove first (i).
Let us choose an edge $e$, such that the deletion of $e$ does not disconnect the graph $G$. Such an edge always exists, since we can pick an edge on a cycle or a leaf edge.
Then, by the minimal transversal-freeness, after the deletion of $e$ we must have a connected subgraph $T$ of $k$ edges such that no disjoint connected
subgraph of $k$ edges exists. This means that if we delete in addition the edges of $T$, then the
graph decomposes into connected subgraphs, each having at most $k-1$
edges. Since these connected subgraphs intersect the vertex set of $T$
in disjoint sets, and since $T$ has at most $k+1$ vertices, we cannot have
more than $(k-1)(k+1) +k +1 =k^2+k$ edges in $G$.

For (ii) let us repeat the same argument and note that in each connected component now we cannot have a tree of $k$ edges. This means that each connected component has at most $k$ vertices,
that is, at most $\binom{k}{2}$ edges, since $G$ is assumed to be simple. Thus, in total we get that
\[
|E| \leq 1+k+(k+1)\binom{k}{2}=\frac{k^3}{2}+\frac{k}{2}+1.
\]

For (iii) we provide a proof for $\cF_{v,c}(G,k)$. The same proof works for $\cF_{v,t}(G,k)$, as well.

Let $v$ be a vertex in $G$ such that $G-v$ is connected.
If $G-v$ contain no connected $k$-edge subgraph, then
we have $|U| \leq k+1$.
Otherwise, since $\cF_{v,c}(G,k)$ is minimal transversal-free, there exists a hyperedge
 $F_v \in  \cF_{v,c}(G,k)$ that intersects all $F \in \cF_{v,c}(G,k)$ with $v \not\in F$.
Let $C_i$ $(i=1, \dots , p$) be connected components of $G-(\{v\} \cup F_v)$.
Then we have $|V(C_i)| \leq k$, since $C_i$ contains at most $k-1$ edges by the definition of $F_v$ and the hypergraph $\cF_{v,c}(G,k)$. Furthermore, we have $N_G(C_i) \subseteq \{v\} \cup H_v$ for all $i$, where $N_G(C_i)$ denotes the set of neighbors of $C_i$ in $G$.

For any component $C_i$,
let $w$ be a vertex in $C_i$.
We first claim that  a hyperedge $F_w$ satisfies
\begin{equation}
\label{eq--123-0}
N_G(C_i) \not \subseteq F_w,
\end{equation}
where we recall that $F_w$ is a hyperedge in
$\cF_{v,c}(G,k)$ that intersects all $F \in \cF_{v,c}(G,k)$ with $w \not\in F$.
Since $\cF_{v,c}(G,k)$ is minimal transversal-free,
$\cF_{v,c}(G,k)$ contains a hyperedge that is disjoint from $F_w$ and contains $w$. This implies the claim.

Let $u$ be a vertex in $N_G(C_i) \setminus F_w$,
Then it holds that
\begin{equation}
\label{eq--123-1}
|N_G(u) \cap (\bigcup_j C_j)| \leq 2k,
\end{equation}
since otherwise,
$|N_G(u)  \setminus (\{w\} \cup F_w) |\geq k$, implying that
$\cF_{v,c}(G,k)$ contains a hyperedge $F$ disjoint from $\{w \} \cup F_w$,
a contradiction.

By \raf{eq--123-0}  and \raf{eq--123-1},
the number of connected components $C_i$ is bounded by $2k|\{v \} \cup F_v | =4k^2+2k$.
Thus, the number of vertices of $G$ is bounded by
\[
(4k^2+2k)k+k+2=4k^3+2k^2+k+2.
\]
\qed

The above bounds imply that for any given $k$ we have only finitely many different such JM hypergraphs, with all four definitions.
The examples derived from trees and stars of cliques show that bounds (i) and (ii) are sharp.

\section{Further Examples and Concluding Remarks}
\label{sec-cr}
Let us first show a small example for which property (A1) holds, but both properties (D1) and (D2) fail.
This example is not JM, showing that not all minimal transversal-free hypergraphs are JM.
Unfortunately, we cannot prove the necessity of properties (D1) or (D2),
though property (D2) may be necessary for a hypergraph to be JM.

Our example is $\cH_{cube}$ formed by the facets of the $3$-dimensional unit cube. The vertex set is $V=\{0,1\}^3$, and the six hyperedges of $\cH_{cube}\subseteq 2^V$ are the subsets $H_{i,\alpha} =\{\sigma\in V\mid \sigma_i=\alpha\}$ for $i=1,2,3$ and $\alpha=0,1$.
This is a $4$-uniform hypergraph on $8$ vertices, and it is clearly minimal transversal-free.
On the other hand it does not satisfy any of three properties (D1), (D2), (D3).
To see that it is not a JM hypergraph,
assume that $m=\binom{3p+1}{2}$ and $q=m+p$ for some positive integer $p$, and consider the position $x\in\ZZP^V$ defined as
$x_{000}=m$, $x_{100}=x_{010}=x_{001}=q$, $x_{110}=x_{101}=x_{001}=2q$ and $x_{111}=3q$.
It is easy to see that for this position we have $m(x)=m$, $y(x)=3p+1$ and $h(x)=3q$, consequently this is a long 
position.
Furthermore, every height move $x \to x'$ is an $H_{i,1}$-move for some $i=1,2,3$ and we must have $m(x')=m(x)$ and $y(x')<y(x)$.
Consequently, $x'$ is always a short 
position. Hence, the necessary property (B1) with $z=h(x)-1$ is violated, and thus,
$\cH_{cube}$ cannot be JM. We show a smallest such position with $p=1$ in Figure \ref{f-cube}.

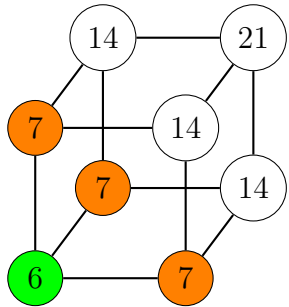
\begin{figure}[htbp]
\centering
\begin{tikzpicture}[z={(0.45cm,0.6cm)},x={(1cm,0cm)},y={(0cm,1cm)},scale=2]
\node[draw,circle,fill=green] (v000) at (0,0,0) {$6$};
\node[draw,circle,fill=orange] (v100) at (1,0,0) {$7$};
\node[draw,circle,fill=orange] (v010) at (0,1,0) {$7$};
\node[draw,circle,fill=orange] (v001) at (0,0,1) {$7$};
\node[draw,circle,fill=white] (v110) at (1,1,0) {$14$};
\node[draw,circle,fill=white] (v101) at (1,0,1) {$14$};
\node[draw,circle,fill=white] (v011) at (0,1,1) {$14$};
\node[draw,circle,fill=white] (v111) at (1,1,1) {$21$};

\path[thick]
    (v000) edge (v100)
    (v000) edge (v010)
    (v000) edge (v001)
    (v100) edge (v110)
    (v100) edge (v101)
    (v010) edge (v011)
    (v010) edge (v110)
    (v001) edge (v101)
    (v001) edge (v011)
    (v111) edge (v110)
    (v111) edge (v101)
    (v111) edge (v011);
\end{tikzpicture}
\caption{A long
position of $\cH_{cube}$ that shows that it is not JM. \label{f-cube}}
\end{figure}

Let us remark that by the definitions for any JM hypergraph $\cH\subseteq 2^V$ the height and the SG functions differ.
Furthermore, JM hypergraphs are minimal for this property
in the sense that for any proper induced subhypergaph $\cH_S$, $S\subset V$ the SG function of $NIM_{\cH_S}$
is the height function of $h_{\cH_S}$. Let us remark that the cube considered above is also minimal for this property.
We refer the reader to the companion paper \cite{BGHMM16}
for more information on height functions.

\end{document}